\documentclass[12pt, leqno]{article}

\usepackage{amsmath}
\usepackage{amsthm}
\usepackage{amssymb}
\usepackage{latexsym}
\usepackage{graphicx}
\allowdisplaybreaks[1]
\usepackage{amsfonts}
\usepackage{ascmac}
\usepackage{enumerate}
\usepackage{bm}
\usepackage{ulem}

\theoremstyle{definition}
\theoremstyle{plain}
\allowdisplaybreaks[1]
\date{}
\newtheorem{Thm}{Theorem}[section]
\newtheorem{Prop}[Thm]{Proposition}
\newtheorem{Lemma}[Thm]{Lemma}
\newtheorem{Cor}[Thm]{Corollary}

\newcommand{\p}{\partial}

\newcommand{\dis}{\displaystyle}

\newcommand{\norm}{\parallel}

\newcommand{\N}{{\mathbb N}}
\newcommand{\R}{{\mathbb R}}

\newcommand{\ep}{\varepsilon}

\newcommand{\2}{\frac{1}{2} }
\newcommand{\wto}{\rightharpoonup}

\newcommand{\meas}{\,{\rm meas}}
\newcommand{\Lloc}{L^1_{t\text{-loc}}}

\title{\bf Co-moving volumes and Reynolds transport theorem in DiPerna-Lions theory}
\author{Kohei Soga\footnote{Department of Mathematics, Faculty of Science and Technology, Keio University, 3-14-1 Hiyoshi, Kohoku-ku, Yokohama 223-8522  Japan. E-mail: soga@math.keio.ac.jp}}

\begin{document}
\maketitle
\begin{abstract} 
\noindent Co-moving volumes and Reynolds transport theorem along a fluid flow are fundamental ingredients for deriving balance laws in fluid mechanics, where the classical theory of  flow maps of ODEs associated to smooth vector fields plays a central role.  
In connection with weak solutions of the Navier-Stokes equations, DiPerna-Lions (Invent. Math. 1989) generalized the classical notions of ODEs and flow maps to vector fields in Sobolev spaces. 
DiPerna-Lions theory also characterizes the evolution of the measure of  the inverse image of each Borel measurable set under generalized flow maps in terms of the divergence of vector fields. 
On the other hand, the image of a measurable set under generalized flow maps, which corresponds to a co-moving volume in the classical theory, is not necessarily measurable, even though the flow map is measurable.
In this paper, we show that if a Borel measurable set is trimmed by removing a null set depending on the pair of evolution time and initial time but not on the initial set, its image becomes Borel measurable, and its measure is consistent with the measure of the classical untrimmed case  in the smooth approximation limit. 
Then, by defining co-moving volumes with such a trimming procedure, we prove a version of Reynolds transport theorem for generalized flow maps. While the usual DiPerna-Lions theory mainly treats transport through compositions with flow maps and push-forward measures, our proofs are organized with the direct treatment of forward images in mind. The approach makes the passage from classical material transport to generalized flows explicit. In particular, it clarifies the measurability of suitably trimmed forward images and their stability under smooth approximation. We also give a detailed construction of generalized flow maps on a bounded domain, using elementary weak compactness arguments in $L^2$ and explicitly treating the roles of the two time variables and the choice of measurable representatives. 

\medskip

  \medskip

  
\end{abstract}
\setcounter{section}{0}
\setcounter{equation}{0}
\section{Introduction}

In fluid dynamics, the mathematical description of fluid motion begins with the assumption that there exists a fluid flow map $X(s,t,\xi)$, where $X(s,t,\xi)$ denotes  the position at time $s$  of the fluid element that is located at position $\xi$ at time $t$, satisfying the following properties for all $s,t,\tau, \xi$: 
\begin{align*} 
&X(s,t,\cdot)\mbox{ is a diffeomorphism},\quad X(t,t,\xi)\equiv\xi,\quad X(s,\tau,X(\tau,t,\xi))=X(s,t,\xi).
\end{align*}
The velocity of each fluid element is then defined as 
\begin{align}\label{ino}
\tilde{v}(s,t,\xi):=\frac{\p}{\p s}X(s,t,\xi),
\end{align}
from which the velocity field on the fluid domain is induced as  
$$v(s,x):=\tilde{v}(s,0,X(0,s,x)).$$  
Note that,  due to the group property of $X$, it holds that $X(s,t,\cdot)^{-1}=X(t,s,\cdot)$ and $\tilde{v}(s,t,\xi)=\tilde{v}(s,\tau,X(\tau,t,\xi))$ for all $s,t,\tau,\xi$. 
Hence, we have $v(s,x):=\tilde{v}(s,0,X(0,s,x))=\tilde{v}(s,t,X(t,s,x))$ for any $t$, i.e., we may choose $0$ as the reference time without loss of generality. 
Furthermore, if we set  $\gamma(s):=X(s,t,\xi)$ for each fixed $(t,\xi)$, we have $\xi= X(t,s,\gamma(s))$ and  by \eqref{ino},   
\begin{align}\label{ODE0}
\gamma'(s)=v(s,\gamma(s)), 
\end{align}
i.e., the trajectory $\gamma(\cdot)$ of the fluid element with initial condition $(t,\xi)$ satisfies \eqref{ODE0} and $X$ is nothing but the flow map of this ODE. 
In the literature, $X$ is often called a Lagrangian flow map, $\tilde{v}$ a Lagrangian velocity and  $(s,\xi) $ the Lagrangian coordinate, while $v$ is called an Eulerian velocity field and $(s,x)$ the Eulerian coordinate. 
We refer to page 5 of  Landau-Lifshitz \cite{LL} and Section 3.12 of Bennett  \cite{AB} for a historical remark that Leonhard Euler was the first one who formulated fluid motion in terms of both coordinates $(s,\xi)$ and $(s,x)$, and  the terminology ``Lagrangian'' in this context is not necessarily consistent with the historical development.   

The governing equations for $v$ and other physical quantities, such as the density $\rho$, the pressure $p$, etc.,  are derived as balance laws within co-moving volumes $X(s,0,A)$, e.g., the momentum balance can be written as 
\begin{align}\label{MB}
\frac{d}{ds}\int_{X(s,0,A)}\rho(s,x)v(s,x)dx=[\mbox{total force acting on $\overline{X(s,0,A)}$}].
\end{align}
Reynolds transport theorem (see Section 2 for a precise formulation) provides a formula to switch $\frac{d}{ds}$ and $\int_{X(s,0,A)}$, which leads to pointwise expressions of balance laws in the limit of $A$ shrinking to a single point. 
For instance,  \eqref{MB} for  a homogeneous incompressible fluid flow forced by a body force $b$ leads to the Navier-Stokes equations   
\begin{align}\label{NS}
 \p_s v(s,x)+ \Big(v(s,x)\cdot\nabla\Big) v(s,x)
=-\frac{1}{\rho}\nabla p(s,x)+\nu\Delta v(s,x)+b(s,x).
\end{align}
If \eqref{NS} admits a smooth solution $v$, the fluid flow map $X$ can be  recovered from the classical ODE theory applied to \eqref{ODE0}. 

 Foias-Guillop\'e-Temam \cite{FGT} justified the ODE \eqref{ODE0} and its flow map $X$ associated with $v$ being a certain Leray-Hopf weak solution of \eqref{NS}, which provided an important precursor to DiPerna-Lions theory.  
 DiPerna-Lions \cite{DL} established the existence of a flow map of  \eqref{ODE0}  generated by $v$ belonging to $L^1([0,T];W^{1,1}_{\rm loc}(\R^n))$ with (the spatial divergence) $\nabla\cdot v\in L^1([0,T];L^\infty(\R^n))$ and $\frac{|v(t,x)|}{1+|x|}\in L^1([0,T];L^1(\R^n))+L^1([0,T];L^\infty(\R^n))$, where the smooth approximation of $v$ and the linear transport equation play a central role in the analysis.    
 In this paper, we study DiPerna-Lions theory in a bounded domain in $\R^3$,  focusing on generalized flow maps. 
The bounded domain case was briefly mentioned in the original work \cite{DL}. Boundary value problems and traces for transport and continuity equations have been studied by Bardos \cite{Ba}, Boyer \cite{Bo}, Crippa-Donadello-Spinolo \cite{CDS} and Crippa-De\,Rosa-Inversi-Nesi \cite{CDIN}. A local theory of maximal regular flows was developed by Ambrosio-Colombo-Figalli \cite{ACF}.
Under the no-slip condition assumed below, zero extension allows one to apply the general theory for the whole space setting. Nevertheless, the bounded domain setting allows us to carry out the compactness argument for the smooth approximation by elementary weak compactness in $L^2$. The commutator estimate and an $L^2$-energy identity then yield uniqueness and strong stability. We give the details, including the different roles of the two time variables in the ODE and the transport equation and the choice of measurable representatives for every fixed time pair in the $L^2$ sense, as preparation for the forward-image analysis.

Throughout this paper, we assume the following\footnote{$ L^p_{t\text{-loc}}(\R\times\Omega):=\{u:\R\times\Omega\to\R\,|\,\mbox{measurable, } \mbox{$u|_{t\in[T_0,T_1]}\in L^p([T_0,T_1]\times\Omega)$ for each $T_0<T_1$}\}.$}
\footnote{The assumption \(\meas(\Omega)=\meas(\overline{\Omega})\) is used, for example, in the proof of Theorem \(\ref{FL-preservative2}\), to show that \(\meas(X^k(t,s,A)\setminus\Omega)\to0\) as \(k\to\infty\). Indeed, the supports of \(v^k\) lie in shrinking neighborhoods of \(\overline{\Omega}\).}: 
\begin{align}\label{v}
\begin{cases}
&\mbox{$\Omega\subset \R^3$ is a bounded connected open set such that $\meas(\Omega)=\meas(\overline{\Omega})$}, \\
&v=(v_1,v_2,v_3)\in \Lloc(\R\times\Omega)^3,\quad \mbox{where $\norm v\norm_{L^1([T_0,T_1]\times\Omega)^3}=\dis\sum_{i=1}^3\norm v_i\norm_{L^1([T_0,T_1]\times\Omega)}$},\\
&\mbox{there exist the weak derivatives $\p_{x_j} v_i \in  \Lloc(\R\times\Omega)$ for $i,j=1,2,3$},  \\
&\mbox{$v_i(t,\cdot)\in W^{1,1}_0(\Omega)$ for a.e. $t\in\R$ and $i=1,2,3$, \quad $\nabla\cdot v\in L^1_{\rm loc}(\R;L^\infty(\Omega))$}.
\end{cases}
\end{align}
We always work with $(t,x)$-measurability of functions and we do not write $v\in L^1_{\rm loc}(\R;W^{1,1}_0(\Omega))^3$.  
Note that $W^{1,1}_0(\Omega)$ corresponds to the no-slip boundary condition in fluid mechanics; for simplicity, we do not consider the non-penetration boundary condition: $v(s,x)\in {\rm T}_x\p\Omega$ for each $(s,x)\in\R\times\p\Omega$.     
 \begin{Thm}[DiPerna-Lions theorem on a bounded domain]\label{Thm-DL}
 Suppose \eqref{v}. Then, there exists a flow map $X:\R\times\R\times\Omega\to\overline{\Omega} $ associated with the ODE $\gamma'(s)=v(s,\gamma(s))$ in the following sense{, with $v$ extended by zero and $X(s,t,x):=x$ for $x\notin\Omega$}:  
 \begin{enumerate}
 \item For each fixed $t\in\R$, the map $X(\cdot,t,\cdot):\R\times\Omega\to\overline{\Omega} $ is measurable and $v(\cdot,X(\cdot,t,\cdot))$ belongs to $ L^1_{s\text{-loc}}(\R\times\Omega)^3$. 

 \item It holds that  for all $s,t,\tau\in\R$ and each  Borel measurable set $A\subset\Omega$,
\begin{align}\label{pre1-2}
&X(t,s, X(s,t,x))=x\quad \mbox{ for a.e. $x\in\Omega$},\\\label{pre2-2}
 & X(s,t,x)=X(s,\tau, X(\tau,t,x)) \quad \mbox{for  a.e. $x\in\Omega$},\\
\label{pre3-3}
 & \frac{1}{c^{s,t}}\meas(A)\le\meas(X(s,t,\cdot)^{-1}(A))\le c^{s,t}\meas(A),
\end{align}
where  $X(s,t,\cdot)^{-1}(A)=\{x\in\Omega\,|\,X(s,t,x)\in A\}$ stands for the inverse image and $c^{s,t}:=e^{\norm \nabla\cdot v\norm_{L^1(I^{s,t};L^\infty( \Omega))}}$ with  $I^{s,t}=[s,t]$ or $[t,s]$. 

 \item For each $t\in\R$ and a.e. $x\in\Omega$, 
 the function $v(\cdot,X(\cdot,t,x))$ is locally integrable  and there exists a null set $N(t,x)\subset\R$ for which\footnote{It follows from this statement that there exists an absolutely continuous curve $\gamma(\cdot)$ that is equal to $X(\cdot,t,x)$  on $ \R\setminus N(t,x)$ and satisfies  $\gamma(s)=x+\int_t^s v(r,\gamma(r))dr$ for all $s\in\R$ or $\gamma'(s)=v(s,\gamma(s))$ a.e. $s\in\R$ with $\gamma(t)=x$.}  
 $$X(s,t,x)=x+\int_t^s v(r,X(r,t,x))dr\quad\mbox{ for all $s\in \R\setminus N(t,x)$}.$$
 \end{enumerate}
 \end{Thm}
We refer to Ambrosio \cite{A1}, \cite{A2} for further developments in the theory of flow maps associated to nonsmooth vector fields including BV-regularity. {The basic properties and stability of these flows are also discussed in Ambrosio-Crippa \cite{AC}. Theorem \ref{Thm-DL} is a bounded-domain formulation of this established theory.} 
 Theorem \ref{Thm-DL} means that for each initial time $t\in\R$ and for a.e. initial position $x\in\Omega$, there exists at least one absolutely continuous solution of  the ODE \eqref{ODE0}, and one can collect these solutions to form a measurable map $X$.  
DiPerna-Lions proved that such collection is unique, but this does not imply (a.e.)  pointwise uniqueness of  \eqref{ODE0}. 
Almost everywhere pointwise uniqueness was first proven by Robinson-Sadowski \cite{RS} in the case of $v$ being a suitable Leray-Hopf weak solution to unforced incompressible Navier-Stokes equations, while pointwise  nonuniqueness on a set of positive measure was first established in  Bru\'e-Colombo-De\,Lellis \cite{BCL}; we refer to Caravenna-Crippa \cite{CC}, Pitcho-Sorella \cite{PS}, Kumar \cite{K},  Galeati \cite{G} and references therein for more on pointwise uniqueness/nonuniqueness of ODEs associated to
nonsmooth vector fields. 
Recently, Lions-Seeger \cite{LS} developed theories of the transport equation and ODEs  associated to one-sided Lipschitz vector fields.   

The main focus of this paper is to analyze the behavior of co-moving volumes and Reynolds transport theorem  in DiPerna-Lions theory. 
{While the inverse image of a Borel measurable set under $X$ is measurable, the measurability of the image $X(s,t,A)$ is not guaranteed by the measurability of $X$ alone\footnote{Throughout this paper, if we just say ``measurable'', it means Lebesgue measurable.}. Thus, co-moving volumes defined directly as $X(s,t,A)$ require care.} 
Note that in the classical case  $X(s,t,A)$ is measurable for each Borel measurable set $A$, as $X(t,s,X(s,t,A))=A$ and $X(s,t,A)=X(t,s,\cdot)^{-1}(A)$ is measurable. 
{We follow the forward-image formulation and trim $A$ so that we can find more precise measure-theoretic properties of the image.}    
Lemma \ref{FL-image} shows that {\it for each measurable set $A\subset\Omega$, there exists a null set $N_{A}^{s,t}\subset \Omega$ such that 
 $X(s,t,A\setminus N_{A}^{s,t})$ is a countable union of closed sets, and hence {Borel measurable}}. 
Nevertheless, $X(s,t,A\setminus N_{A}^{s,t})$ could still exhibit a measure evolution that is not consistent with the limit of the classical case  obtained in the smooth approximation of $X$.  
Such difficulty is found at the points for which  \eqref{pre1-2} fails, i.e.,  
the null set 
 $$E^{s,t}:=\Omega\setminus\{x\in\Omega\,|\,  X(t,s, X(s,t,x))=x  \}.$$
{After \eqref{pre1-2} has been established, the same Egorov argument can be applied to $\Omega\setminus E^{s,t}$ instead of each initial set. In Proposition \ref{FL-common}, we show that, for every fixed $s,t\in\R$, a null set $\dot{N}^{s,t}\supset E^{s,t}$ can be chosen independently of $A$ so that $X(s,t,A\setminus\dot{N}^{s,t})$ is Borel measurable for every Borel measurable set $A\subset\Omega$.}     
With this trimming technique, we define ``co-moving volumes'' associated to $X$, where initial sets are assumed to be Borel measurable so that classical flow maps are applicable without any trimming.   
\medskip

\noindent{\bf Definition 1.} {\it For each $s,t\in\R$, let $\dot{N}^{s,t}\subset\Omega$ denote a null set containing $E^{s,t}$ such that $X(s,t,A\setminus\dot{N}^{s,t})$ is Borel measurable for every Borel measurable set $A\subset\Omega$. For a given Borel measurable set $A\subset\Omega$, we call the family of sets $\{X(s,t,A\setminus\dot{N}^{s,t})\}_{s\in\R}$ a regular co-moving volume starting from $t\in\R$.}
The null set $\dot{N}^{s,t}$ is common to the initial sets $A$, but may depend on $s$ and $t$.

\medskip
 \noindent Although the trimming is not unique,  all such regular co-moving volumes with respect to $A$ have the same measure evolution that is characterized by the limit of smooth approximation $\{X^k\}_{k\in\N}$ of $X$, where each $X^k$ is a smooth flow map associated to the mollification of $v$ with zero-extension outside $\Omega$ (see the definition of $X^k$ stated with \eqref{LT-ODE} in Section 3).  
\begin{Thm}[Stability of regular co-moving volumes]\label{FL-preservative2}
Let  $A\subset\Omega$ be  Borel measurable  and  $\{X(s,t,A\setminus {\dot{N}^{s,t}} )\}_{s\in\R}$ be  a regular co-moving volume as defined in Definition 1. Then, for each $s\in\R$ and $f\in L^1(\R^3)$, it holds that 
\begin{align}\label{RTT1}
&\mbox{$\dis\meas(X(s,t,\cdot)^{-1}(A))=\lim_{k\to\infty} \meas(  X^k(s,t,\cdot)^{-1}(A))=\lim_{k\to\infty} \meas(X^k(t,s,A))$},\\\label{RTT2}  
&\mbox{$\dis \meas(X(s,t,A\setminus {\dot{N}^{s,t}}))=\lim_{k\to\infty} \meas(X^k(s,t, A))=\meas( X(t,s,\cdot)^{-1}( A))$},\\\label{RTT3} 
& \mbox{$\dis \int_{X(s,t,A\setminus {\dot{N}^{s,t}})} f(x)dx= \lim_{k\to\infty}  \int_{X^k(s,t,A)} f(x)dx$},\\\label{RTT4} 
& \mbox{$\dis \int_{X(s,t,\cdot)^{-1}(A)} f(x)dx= \lim_{k\to\infty}  \int_{X^k(s,t,\cdot)^{-1}(A)} f(x)dx= \lim_{k\to\infty}  \int_{X^k(t,s,A)} f(x)dx$},\\\label{RTT5}
&{X(s,t,A\setminus\dot{N}^{s,t})\subset X(t,s,\cdot)^{-1}(A).}
\end{align}
\end{Thm}
As a geometrical consequence of Theorem \ref{FL-preservative2},  
we see that a set transported by the smooth approximate flows
converges to the regular co-moving volume. 
This may be useful in investigation of  moving domains by smooth approximation.
\begin{Cor}\label{FL-preservative2-2}
Under the assumptions of Theorem \ref{FL-preservative2}, for any Borel measurable set $A\subset\Omega$ and any fixed $s,t\in\mathbb R$, it holds that\footnote{Equivalently, $\meas\Big(  \big\{X^k(s,t,A)\cup X(s,t,A\setminus{\dot{N}^{s,t}})\big\}
 \setminus\big\{X^k(s,t,A)\cap X(s,t,A\setminus{\dot{N}^{s,t}})\big\}\Big)\to0$ as $k\to\infty$.} 
$$
\chi_{X^k(s,t,A)}\to \chi_{X(s,t,A\setminus{\dot{N}^{s,t}})}
\quad\text{ in $L^1(\mathbb R^3)$ as $k\to\infty$,}
$$
where $\chi_D$ stands for the indicator function of the set $D$.  
\end{Cor}
\noindent In fact, by \eqref{RTT2}  and  \eqref{RTT3}, with
$f:=\chi_{X(s,t,A\setminus{\dot{N}^{s,t}})}\in L^1(\mathbb R^3)$ ({\it note that this requires the measurability of the co-moving domain}), we have as $k\to\infty$,
\begin{align*}
&\meas(X^k(s,t,A))
\to \meas(X(s,t,A\setminus{\dot{N}^{s,t}})),\\
&\meas\bigl(X^k(s,t,A)\cap X(s,t,A\setminus{\dot{N}^{s,t}})\bigr)
=\int_{X^k(s,t,A)}
\chi_{X(s,t,A\setminus{\dot{N}^{s,t}})}(x)\,dx\\
&\quad\to 
\int_{X(s,t,A\setminus{\dot{N}^{s,t}})}
\chi_{X(s,t,A\setminus{\dot{N}^{s,t}})}(x)\,dx
= \meas(X(s,t,A\setminus{\dot{N}^{s,t}})),
\end{align*}
from which we obtain   as $k\to\infty$,
\begin{align*}
&\norm\chi_{X^k(s,t,A)}-
\chi_{X(s,t,A\setminus{\dot{N}^{s,t}})}\norm_{L^1(\mathbb R^3)}\\
&=\meas(X^k(s,t,A))
+\meas(X(s,t,A\setminus{\dot{N}^{s,t}}))
-2\meas\bigl(X^k(s,t,A)\cap X(s,t,A\setminus{\dot{N}^{s,t}})\bigr)
\to 0.
\end{align*}
\indent Furthermore, Reynolds transport theorem holds as a consequence of Theorem \ref{FL-preservative2}:  
\begin{Thm}[Reynolds transport theorem in DiPerna-Lions theory]\label{main} 
Let $A\subset\Omega$ be Borel measurable and $\{X(s,t,A\setminus {\dot{N}^{s,t}} )\}_{s\in\R}$ be a regular co-moving volume as defined in Definition 1. 
Then, for each  $s\in\R$ and $C^1$-smooth function $g:\R\times \R^3\to\R$, it holds that
\begin{align}\label{trans2}
&\int_{X(s,t,A\setminus {\dot{N}^{s,t}})}\!\!\! g(s,x)dx=\int_{A}g(t,x)dx+\int_t^s\int_{X(r,t,A\setminus {\dot{N}^{r,t}})}\left\{\frac{\partial g}{\partial r}(r,x)+\nabla\cdot\Big(g(r,x)v(r,x)\Big)\right\}dxdr,\\\label{trans3}
&\int_{X(s,t,\cdot)^{-1}(A)}\!\!\! g(t,x)dx=\int_{A}g(s,x)dx+\int_s^t\int_{X(s,r,\cdot)^{-1}(A)}\left\{\frac{\partial g}{\partial r}(r,x)+\nabla\cdot\Big(g(r,x)v(r,x)\Big)\right\}dxdr,\\\label{trans1}
&\meas(X(s,t,A\setminus {\dot{N}^{s,t}})) =\meas(A) +\int_t^s\int_{X(r,t,A\setminus {\dot{N}^{r,t}})} \nabla\cdot v(r,x)dxdr,\\
\label{trans0}
&\meas(X(s,t,\cdot)^{-1}( A))=\meas(A)+\int_s^t\int_{X(s,r,\cdot)^{-1}(A)} \nabla\cdot v(r,x)dxdr. 
\end{align}
\end{Thm} 
\medskip

\noindent {\bf Remark.} {\it 
The formulas \eqref{trans2} to \eqref{trans0} arise as limits of classical Reynolds transport theorem and Liouville theorem with respect to the set $A$ (without any trimming) under the smooth flows $X^k$ associated to the mollification of  $v$; in particular, due to Theorem \ref{FL-preservative2}, each term of the formulas \eqref{trans2} and \eqref{trans1} is independent of the choice of trimming $ {\dot{N}^{r,t}}$; the a.e.-differential forms of  \eqref{trans2} to \eqref{trans0} are also available. }
\medskip
    
The focus of Theorems \ref{FL-preservative2} and \ref{main} is the direct treatment of forward images, which correspond to co-moving volumes in classical fluid mechanics. The usual DiPerna-Lions-Ambrosio theory mainly describes transport through compositions with flow maps and push-forward measures. Our approach makes the passage from classical material transport to generalized flows explicit. The Remark following the proof of Theorem \ref{FL-preservative2} identifies the trimmed forward images with the inverse images under the backward flow, up to null sets.

    Our proofs of  Theorem \ref{FL-preservative2} and Theorem \ref{main}  rely essentially on the fact that a sequence of approximate flows converges strongly to $X$ in $L^1$, and that this sequence contains an a.e. pointwise convergent subsequence, where Egorov's theorem, applied to this subsequence, is crucial in the argument; {we then identify the integral limit and recover convergence of the whole sequence.}    
Such $L^1$-convergence of approximate flows has been well-established in the whole space setting by DiPerna-Lions \cite{DL} through the uniqueness and renormalization property of weak solutions of the linear transport equation. 
In the bounded domain setting,  compactness arguments can be more elementary. 
However, mollification arguments require extension of functions outside the original domain, which causes leakage of trajectories; this issue must be handled carefully.  
To keep the paper reasonably self-contained and to clarify the highly technical and sometimes delicate aspects of DiPerna-Lions theory, we will prove  Theorem \ref{FL-preservative2} and Theorem \ref{main}  together with Theorem \ref{Thm-DL}  in  full detail.  

We conclude the introduction by referring to a recent work by Bothe-K\"ohne \cite{BK} (see also Bothe \cite{B1}) on an alternative generalization of  Reynolds transport theorem for a flow map generated by a discontinuous vector field arising from a two-phase flow with a sharp interface and phase change. 
Even though the velocity field is assumed to be  jointly continuous and locally  Lipschitz in the spatial variable in each phase, discontinuity across the interface requires the notion of differential inclusions. If a co-moving volume intersects the interface, a surface integral remains in the transport formula due to discontinuity across the interface. 

\setcounter{section}{1}
\setcounter{equation}{0}
\section{Classical theory of ODEs and the  linear transport equation}

This section is devoted to an overview of the classical theory of ODEs and linear transport equation, which will be applied to the problem arising from smooth approximation of $v$ introduced in \eqref{v}.  

 Let $K\subset \R^3$ be a bounded connected open set  and $w:\R\times\R^3\to\R^3$ be $C^1$-smooth satisfying $w(t,x)|_{x\in \R^3\setminus K}\equiv0$, where $K$ is later taken so that $K\supset \overline{\Omega}$, and the mollification of the zero extension of $v$ will serve as such a vector field 
$w$. 
Consider the ODE 
\begin{align}\label{1ODE}
x'(s)=w(s,x(s)),\,\,\,\quad x(t)=\xi, 
\end{align}  
where $(t,\xi)\in\R\times\R^3$ is arbitrary. 
Since $w|_{\R^3\setminus K}\equiv0$, $x(\cdot)$ stays in $\R^3\setminus K$ and remains stationary if $\xi\in\R^3\setminus K$. Hence,  $x(\cdot)$ stays in $K$ if $\xi\in K$, which implies that \eqref{1ODE} has a global solution for every initial condition. 
The flow map $X:\R\times\R\times\R^3\to\R^3$, $X(s,t,\xi):=x(s)$ of \eqref{1ODE} is then well-defined to be $C^1$-smooth in all variables.  
In particular, $X(s,t,\cdot): K\to K$ is $C^1$-diffeomorphic,  $X(s,t,\p K)=\p K$ and $X(s,t,\cdot)^{-1}=X(t,s,\cdot)$ for each $t,s\in\R$. 
Our main interest is the restriction of $X$ on $K$. 
 
Next, consider the linear transport equation parametrized by initial time $s\in\R$:  
\begin{align}\label{1LT}
\begin{cases}
\p_t \rho(s,t,x)+w(t,x)\cdot\nabla \rho(s,t,x)=0&\mbox{\quad for $(t,x)\in\R\times K$},\\
\rho(s,s,x)=\rho_0(x)&\mbox{\quad for $x\in K$}, 
\end{cases}
\end{align}
where $\rho_0$ is a $C^1$-smooth function and $\nabla=\frac{\p}{\p x}$. 
The following remark would be helpful: {\it the evolution time $s$ (resp. initial time $t$) for the ODE  \eqref{1ODE} corresponds to the initial time (resp. evolution time) for the PDE \eqref{1LT}.}
\begin{Prop}\label{classical1}
Define $\rho: \R\times\R\times K\to\R$ as  
$\rho(s,t,x):=\rho_0 (X(s,t,x))$. Then, $\rho(s,\cdot,\cdot)$ is the unique $C^1$-solution of \eqref{1LT} for each $s \in\R$. 
\end{Prop}
\begin{proof}
It follows from $X(t,s,X(s,t,x))=x$ that 
\begin{align*}
&\p_\xi X(s,t,x)=\p_\xi X(t,s,X(s,t,x))^{-1},\quad 
\p_t X(s,t,x)=-(\p_\xi X(t,s,X(s,t,x)))^{-1} w(t,x), 
\end{align*}
from which we obtain $\p_t\rho(s,t,x)+w(t,x)\cdot \nabla\rho(s,t,x)=0$ for all $(t,x)\in\R\times K$. 
The initial condition is satisfied as $\rho(s,s,x)=\rho_0(X(s,s,x))=\rho_0(x)$. 

Let  $\rho,\tilde{\rho}$ be two smooth solutions. Then, $u:=\rho-\tilde{\rho}$ is a solution of \eqref{1LT} with $\rho_0=0$. 
Multiplying  $\p_t u(s,t,x)+w(t,x)\cdot \nabla u(s,t,x)=0$ by $u$ and integrating it over $K$, we find that 
\begin{align}\label{uru}
\frac{\p}{\p t}\norm u(s,t,\cdot)\norm_{L^2(K)}^2=\int_Ku({s,t},x)^2 (\nabla\cdot w(t,x))dx\le \max_{x\in \overline{K}}|\nabla\cdot w(t,x)|\times\norm u(s,t,\cdot)\norm_{L^2(K)}^2
\end{align} 
with $\norm u(s,s,\cdot) \norm_{L^2(K)}^2=0$. 
Gronwall's inequality implies that $ \norm u(s,t,\cdot) \norm_{L^2(K)}^2=0$ for all $t\in\R$. 
\end{proof}
If $\rho_0(x)=x_i$ ($i=1,2,3$), the solution of \eqref{1LT} is the $i$-th component of $X(s,t,x)$.  
Hence, the flow map of the ODE can be obtained from solutions to the linear transport equation. This is the core idea of DiPerna-Lions theory to obtain  generalized flow maps.  
\begin{Prop}\label{classical2}
For each $f\in L^1(K)$ and all $s,t\in\R$, it holds that  
$$\int_K f(X(s,t,x))dx=\int_K f(x)dx+\int_s^t \int_Kf(X(s,r,x))(\nabla \cdot w(r,x))dxdr.$$
\end{Prop}
\begin{proof} 
Set $F(s,t):=\int_K f(X(s,t,x))dx$, where $|\det \p_\xi X(s,t,\xi)|$ is continuous and   $\int_K f(X(s,t,x))dx=\int_Kf(y)|\det \p_\xi X(t,s,y)|dy$ is  finite for all $s,t\in\R$. 
Then, since $\det\p_\xi X(t+h,t,y)=1+h\nabla\cdot w(t,y)+o(h)$ for all $|h|\ll1$, we obtain  
\begin{align*}
F(s,t+h)=&\int_K f(X(s,t,X(t,t+h,x)))dx
=\int_K f(X(s,t,y))|\det\p_\xi X(t+h,t,y) |dy\\
=&\int_K f(X(s,t,y)) dy +h\int_K f(X(s,t,y))(\nabla \cdot w(t,y))dy +o(h),
\end{align*}
which leads to $\p_tF(s,t)=\int_K f(X(s,t,y))(\nabla \cdot w(t,y))dy$, where the right hand side is continuous with respect to $t$.  
Since $F(s,s)=\int_K f(x)dx$, we  conclude the assertion.  
\end{proof}
\begin{Prop}[Classical Reynolds transport theorem]\label{classical3}
\label{Rey}
Let $A\subset K$ be Borel measurable.  Then, for every $C^1$-smooth function $g:\R\times \overline{K}\to\R$ and all $s,t\in\R$, the following identities hold: 
\begin{align}\label{usi1}
&\int_{X(s,t,A)}g(s,x)dx=\int_{A}g(t,x)dx+\int_t^s\int_{X(r,t,A)}\left\{\frac{\partial g}{\partial{r}}(r,x)+\nabla\cdot\Big(g(r,x)w(r,x)\Big)\right\}dxdr,\\\label{usi2}
&\meas(X(s,t,A))=\meas(A)+\int_t^s\int_{X(r,t,A)}\nabla\cdot w(r,x)dxdr\quad \mbox{(Liouville theorem)}.
\end{align}
\end{Prop}
\begin{proof} The identity \eqref{usi2} follows from \eqref{usi1} by choosing $g=1$. 
Set  $G(s,t):=\int_{X(s,t,A)}g(s,x)dx$. Let  $A\subset K$ be open. Observe that  
\begin{align*}
G(s+h,t)
&= \int_{X(s+h,t,A)}g(s+h,x)dx
=\int_{X(s,t,A)}g(s+h, X(s+h,s,y))\Big|\det\p_\xi X(s+h,s,y)\Big|dy. 
\end{align*}
Due to $\det\p_\xi X(s+h,s,y)=1+h\nabla\cdot w(s,y)+o(h)$, we have  
\begin{align*}
G(s+h,t)=\int_{X(s,t,A)}g(s+h,X(s+h,s,y))(1+h\nabla\cdot w(s,y)) dy+o(h).
\end{align*}
Therefore, we obtain 
\begin{align*}
&\frac{G(s+h,t)-G(s,t)}{h}\\
&\quad =\int_{X(s,t,A)} \Big\{\frac{g(s+h,X(s+h,s,y))
 -g(s,y)}{h}+g(s+h,X(s+h,s,y))\nabla\cdot w(s,y)\Big\}dy +\frac{o(h)}{h}\\
&\quad \to \int_{X(s,t,A)} \Big\{\p_sg(s,y)+\nabla g(s,y)\cdot \p_s X(s,s,y)+ g(s,y)\nabla\cdot w(s,y)\Big\}dy\quad \mbox{as $h\to0$}.
\end{align*}
Since $\p_s X(s,s,y)=w(s,y)$ and the last integral is continuous with respect to $s$, we conclude \eqref{usi1}. 

When  $A\subset K$ is closed, the regularity of the Lebesgue measure allows us to choose an open set $A^\ep \subset K$ such that $A^\ep\supset A$ and $\meas(A^\ep\setminus A)<\ep$ for each $\ep>0$.  
We apply \eqref{usi2} with the open set  $C^\ep:=A^\ep\setminus A$ to know that $\meas(X(r,t,{C^\ep}))\to 0$ as $\ep\to0$ for all $r$.   
Then, we  apply \eqref{usi1} with the open set $A^\ep =A\cup C^\ep$  and  send  $\ep$ to zero  to conclude \eqref{usi1}.  
Similarly, when $A$ is Borel measurable, we take an open set $A^\ep \supset A$ and a closed set $B^\ep\subset A$ such that $\meas(A^\ep\setminus A)<\ep$, $\meas(A\setminus B^\ep)<\ep$ for each $\ep>0$.  
 We apply \eqref{usi2} with $A^\ep$ and $B^\ep$ to obtain $\meas(X(r,t,A^\ep)\setminus X(r,t,A))\le \meas(X(r,t,A^\ep)\setminus X(r,t,B^\ep))\to 0$ as $\ep\to0$ for all $r$. 
 Then, we  apply \eqref{usi1} with $A^\ep $   and  send  $\ep$ to zero  to conclude \eqref{usi1}.  
\end{proof}     
\setcounter{section}{2}
\setcounter{equation}{0}
\section{Linear transport equation on a bounded domain} 

In this  section, we investigate weak solutions of the linear transport equation on a bounded domain, which play a key role in  the proofs of Theorems \ref{Thm-DL}, \ref{FL-preservative2} and \ref{main}.   
The goal of this section is to obtain a family of weak solutions of  \eqref{LT} parameterized by $s$ and to clarify their $(s,t)$-pointwise interpretation together with their smooth approximation.  
{The results in this section are closely related to the established DiPerna-Lions theory. Nevertheless, it is useful to explicitly describe} the technical details specific to the bounded domain case, such as smooth approximation complicated by  trajectory leakage,  dependence on the initial time, compactness, measurability, the choice of suitable representative elements, etc., which make the proofs of the main theorems significantly clearer. 
{Due to the bounded domain setting, the compactness argument for bounded approximate solutions can be carried out by elementary weak compactness in $L^2$ on finite time intervals. The commutator estimate and the $L^2$-energy identity then yield uniqueness and strong stability through Gronwall's inequality. We also distinguish the roles of the two time variables in the ODE and the transport equation and explain the choice of measurable representatives used below to treat every fixed time pair in the $L^2$ sense.}   
The proofs are presented immediately after the corresponding statements.  

Let $\Omega\subset \R^3$ and $v=(v_1,v_2,v_3)$ be as in \eqref{v}. 
Note that if $v$ is originally defined only on $[0,\infty)\times\Omega$, the extension $\tilde{v}$ of $v$ as $\tilde{v}(t,x):=v(t,x)$ for $t\ge0$ and $\tilde{v}(t,x):=v(-t,x)$ for $t<0$ recovers \eqref{v}.  
Let $K\subset\R^3$ be a bounded  open ball\footnote{For technical convenience, we take a convex set.} that contains $\overline{\Omega}$. 
We first investigate well-posedness of weak solutions to the linear transport equation with initial time $s\in\R$ 
\begin{align}\label{LT}
\begin{cases}
\p_t \rho(s,t,x)+v(t,x)\cdot\nabla \rho(s,t,x)=0&\mbox{\quad for  $(t,x)\in \R\times\Omega $},\\
\rho(s,s,x)=\rho_0(x)&\mbox{\quad for $x\in\Omega$}
\end{cases}
\end{align}   
in the solution class $L^\infty(\R\times\Omega)$, where we give initial data $\rho_0$ from $L^\infty(\Omega)$ independently from $s$ and deal with the solution of \eqref{LT} having the parameter $s$. 
\medskip

\noindent {\bf Definition 2.} {\it For each fixed $s\in\R$, a function $\rho(s,\cdot,\cdot)\in L^\infty(\R\times\Omega)$ is called a weak solution of \eqref{LT}, provided for every  test function $\varphi\in C^\infty_0(\R^4)$   
\begin{align}\label{LT-weak}
&\int_s^{\infty}\int_\Omega\Big\{ \rho(s,t,x)\p_t\varphi(t,x)+v(t,x)\rho(s,t,x)\cdot\nabla\varphi(t,x)\\\nonumber 
&\qquad + (\nabla\cdot v(t,x))\rho(s,t,x)\varphi(t,x)\Big\}dxdt+\int_\Omega\rho_0(x)\varphi(s,x)\,dx=0, 
\\\label{LT-weak2}
&\int_{-\infty}^{s}\int_\Omega \Big\{\rho(s,t,x)\p_t\varphi(t,x)+v(t,x)\rho(s,t,x)\cdot\nabla\varphi(t,x)\\\nonumber
&\qquad+(\nabla\cdot v(t,x))\rho(s,t,x)\varphi(t,x)\Big\}dxdt-\int_\Omega\rho_0(x)\varphi(s,x)\,dx=0.
\end{align}}
\vspace{-3mm}

\noindent {\bf Remark.} {\it We deliberately do not restrict the test functions to be  compactly supported in $\R\times\Omega$; the reason for this choice is explained in the proof of Lemma \ref{LT-commu} below. }
\medskip

Existence of weak solutions is proven through a mollification technique, i.e., we regularize $v$ and $\rho_0$ so that the classical results for  \eqref{1ODE} and \eqref{1LT} are applicable.   
Let $\{\ep_k\}_{k\in\N}$ be a  sequence such that $0<\ep_1\ll1$ and  $\ep_k\searrow 0$ as $k\to\infty$. 
Extend  $v$  to be zero over $\R^3\setminus\Omega$, where the extension is still weakly $x$-differentiable,  and define $v^k:\R\times\R^3\to\R^3$ as the mollification of $v$ by means of the mollifier in $\R^4$ with the parameter $\ep_k$.  
We define $\rho_0^k$ in one of the following two ways: 
\begin{itemize}
\item[(I1)] If initial data $\rho_0$ is given as $\rho_0=f|_{\Omega}$ with some $f\in C^1(\R^3)$, e.g., $\rho_0(x)=x_i$ ($i=1,2$ or $3$), define $\rho^k_0:=f$ for each $k\in\N$. 
\item[(I2)]  Extend $\rho_0$ to be zero over $\R^3\setminus \Omega$ and  define $\rho^k_0:\R^3\to\R$ as the mollification of $\rho_0$ by means of the mollifier in $\R^3$ with the parameter $\ep_k$.  
\end{itemize}
Then it holds that   for all $T_0<T_1$ and $i,j=1,2,3$, 
\begin{align*}
\begin{cases}
v_i^k\in C^\infty(\R\times\R^3), \qquad\,\,\,\, \,
 {\rm supp} (v_i^k)\subset \R\times  K \quad \mbox{for all $k\in\N$},\\ 
\norm v_i^k-v_i\norm_{L^1([T_0,T_1]\times K)}, \quad 
\norm \p_{x_j}v_i^k-\p_{x_j}v_i\norm_{L^1([T_0,T_1]\times K)}\to0\mbox{\,\,\, as $k\to\infty$},\\
\mbox{$\rho_0^k\in C^1(\R^3)$,\qquad\qquad \,\,\,\,\,
$\norm \rho_0^k-\rho_0\norm_{L^1(K)}\to0$ as  $k\to\infty$.}
\end{cases}
\end{align*} 
{In this convergence on $K$, $\rho_0$ denotes $f$ in (I1) and the zero extension in (I2).}

Consider the ODE 
\begin{align}\label{LT-ODE}
x'(s)=v^k(s,x(s)),\quad x(t)=\xi\quad \mbox{for each $(t,\xi)\in\R\times\R^3$}.   
\end{align}
Since $v^k$ vanishes outside $\R\times K$, there exists the $C^1$-flow map $X^k:\R\times\R\times\R^3\to \R^3$ of \eqref{LT-ODE} such that $X^k(s,t,K)=K$, $X^k(s,t,\p K)=\p K$ and $X^k(s,t,\cdot)^{-1}=X^k(t,s,\cdot)$. 
Define the $C^1$-function $\rho^k(s,\cdot,\cdot)$ for each $s\in\R$ as\footnote{$\rho^{k}(s,t,x)$ is actually $C^1$-smooth for all {variables} $s,t,x$.} 
\begin{align}\label{LT-approx}
\rho^k(s,\cdot,\cdot):\R\times \R^3\to\R,\,\,\,\,\, \rho^k(s,t,x):=\rho_0^k(X^k(s,t,x)), 
\end{align}
where $X^k(s,t,\xi)=\xi$ for all $(t,\xi)\in\R\times(\R^3\setminus K)$ due to  ${\rm supp} (v^k)\subset \R\times  K$ and hence $\rho^k(s,t,x)=\rho_0^k(x)$ for all $(t,x)\in\R\times (\R^3\setminus K)$. 
We see that $\rho^k$ satisfies for all $k\in\N$,
\begin{align}\label{1LT-2}
&\begin{cases}
\p_t \rho^k(s,t,x)+v^k(t,x)\cdot\nabla \rho^k(s,t,x)=0&\mbox{\quad for $(t,x)\in\R\times \R^3$},\\
\rho^k(s,s,x)=\rho_0^k(x)&\mbox{\quad for $x\in \R^3$}. 
\end{cases}
\end{align}
A weak solution of \eqref{LT} is obtained as a weak limit of  solutions of \eqref{1LT-2}.
\begin{Prop}\label{LT-existence}
For each $\rho_0\in L^{\infty}(\Omega)$ and $s\in\R$, there exists a weak solution $\rho(s,\cdot,\cdot)\in L^\infty(\R\times\Omega)$ of \eqref{LT}.
\end{Prop}

\begin{proof} 
We investigate weak convergence of the above $\rho^k$ together with the weak form of \eqref{1LT-2}. 
Let $\beta\ge0$ be a constant such that $-\beta\le \rho_0(x)\le \beta$ a.e. $x\in K$.  
Since $K$ is bounded, $\{\rho^k(s,\cdot,\cdot)=\rho_0^k(X^k(s,\cdot,\cdot))\}_{k\in\N}$ is bounded\footnote{One could proceed by $L^\infty$-calculus here. In any case, we later need $L^p$-calculus with $1<p<\infty$ and we demonstrate the current proof through elementary $L^2$-calculus.} in $L^2([-T,T]\times K)$ for each $T>0$.   
A diagonal $L^2$-weak-compactness argument is used to obtain a time-globally convergent subsequence.

Since $\{\rho^k(s,\cdot,\cdot)|_{t\in[-1,1]}\}_{k\in\N}$ is a bounded sequence of $L^2([-1,1]\times K)$,  we find a weakly convergent subsequence $\{\rho^{a^1_k}(s,\cdot,\cdot)\}_{k\in\N}\subset \{\rho^k(s,\cdot,\cdot)\}_{k\in\N}$ and its limit $\mu^1(s,\cdot,\cdot)\in L^2([-1,1]\times K)$. 
We claim that 
$$-\beta\le \mu^1(s,t,x)\le \beta\mbox{ a.e. $(t,x)\in [-1,1]\times K$. }$$
In fact, set the non-positive function $\nu(t,x):=\min\{ \mu^1(s,t,x)+\beta,0 \};$   
 since $\rho^{a^1_k}(s,\cdot,\cdot)+\beta\ge0$ a.e. by assumption, we have $(\rho^{a^1_k}(s,\cdot,\cdot)+\beta, \nu)_{L^2([-1,1]\times K)}\le0$ for all $k\in\N$ and  
 $$(\rho^{a^1_k}(s,\cdot,\cdot)+\beta, \nu)_{L^2([-1,1]\times K)}\to(\mu^1(s,\cdot,\cdot)+\beta, \nu)_{L^2([-1,1]\times K)}=\norm\nu\norm_{L^2([-1,1]\times K)}^2\mbox{ as $k\to\infty$},$$ 
which leads to $\norm\nu\norm_{L^2([-1,1]\times K)}^2\le0$ and $\nu=0$, i.e., $\mu^1(s,\cdot,\cdot)\ge-\beta$; 
similarly, with $\nu(x):=\min\{ \beta- \mu^1(s,t,x),0 \}$, we obtain $\mu^1(s,\cdot,\cdot)\le\beta$. 

We find an $L^2([-2,2]\times K)$-weakly convergent subsequence $\{\rho^{a^2_k}(s,\cdot,\cdot)\}_{k\in\N}\subset \{\rho^{a^1_k}(s,\cdot,\cdot)\}_{k\in\N}$ and its limit $\mu^2(s,\cdot,\cdot)\in L^2([-2,2]\times K)$ such that $-\beta\le \mu^2(s,\cdot,\cdot)\le\beta$.
For any $f\in L^2([-2,2]\times K)$ with supp$(f)=[-1,1]\times K$, we have 
$$(\rho^{a^2_k}(s,\cdot,\cdot),f)_{L^2([-2,2]\times K)}\to (\mu^{1}(s,\cdot,\cdot),f)_{L^2([-1,1]\times K)}=(\mu^{2}(s,\cdot,\cdot),f)_{L^2([-1,1]\times K)}\quad(k\to\infty),$$
which means that $\mu^1(s,\cdot,\cdot)=\mu^2(s,\cdot,\cdot)|_{t\in[-1,1]}$. 
Repeating this process, we find for each $n\in\N$ an $L^2([-n-1,n+1]\times K)$-weakly convergent subsequence $\{\rho^{a^{n+1}_k}(s,\cdot,\cdot)\}_{k\in\N}\subset \{\rho^{a^n_k}(s,\cdot,\cdot)\}_{k\in\N}$ and its limit $\mu^{n+1}(s,\cdot,\cdot)\in L^2([-n-1,n+1]\times K)$  such that $-\beta\le \mu^{n+1}(s,\cdot,\cdot)\le\beta$, where $\mu^{n+1}(s,\cdot,\cdot)|_{t\in[-n',n']}=\mu^{n'}(s,\cdot,\cdot)$ for all $n'\le n$.  
Define $a_k:=a^k_k$ and $\rho(s,\cdot,\cdot)\in L^2_{t\text{-loc}}(\R\times K)$ as 
$$\mbox{$\rho(s,t,x):=\mu^n(s,t,x)$ with $n\in\N$ such that  $t\in[-n,n]$. }$$  
For each $T_0<T_1$, taking $n\in\N$ such that $-n\le T_0<T_1\le n$, we have $\{\rho^{a_k}(s,\cdot,\cdot)\}_{k\ge n}\subset \{\rho^{a^n_k}(s,\cdot,\cdot)\}_{k\in\N}$ and 
\begin{align*}
&\rho^{a_k}(s,\cdot,\cdot)|_{t\in[T_0,T_1]} \wto \mu^n(s,\cdot,\cdot)|_{t\in[T_0,T_1]}=\rho(s,\cdot,\cdot)|_{t\in[T_0,T_1]}  \mbox{ \,\,\, in $L^2([T_0,T_1]\times K)$ as $k\to\infty$}.
\end{align*} 
 \indent Since $\rho^{a_k}(s,\cdot,\cdot)$ is smooth satisfying \eqref{1LT-2} and $\rho^{a_k}(s,t,x)=\rho_0^{a_k}(x)$ for all $(t,x)\in \R\times (\R^3\setminus K)$ due to  $ {\rm supp} (v^{a_k})\subset \R\times  K$, for each test function $\varphi\in C^\infty_0(\R^4)$, integration of \eqref{1LT-2}$\times \varphi$ over $[s,\infty)\times\R^3$ and $(-\infty,s]\times\R^3$ yields 
\begin{align*}
0=&\int_s^{\infty}\int_{K}\Big\{ \rho^{a_k}(s,t,x)\p_t\varphi(t,x)+v^{a_k}(t,x)\rho^{a_k}(s,t,x)\cdot\nabla\varphi(t,x)\\
&+ (\nabla\cdot v^{a_k}(t,x))\rho^{a_k}(s,t,x)\varphi(t,x)\Big\}dxdt+\int_{K}\rho^{a_k}_0(x)\varphi(s,x)\,dx,\\ 
0=&\int_{-\infty}^{s}\int_{K}\Big\{ \rho^{a_k}(s,t,x)\p_t\varphi(t,x)+v^{a_k}(t,x)\rho^{a_k}(s,t,x)\cdot\nabla\varphi(t,x)\\
&\quad + (\nabla\cdot v^{a_k}(t,x))\rho^{a_k}(s,t,x)\varphi(t,x)\Big\}dxdt
-\int_{K}\rho^{a_k}_0(x)\varphi(s,x)\,dx, 
\end{align*}
where $\int_s^{\infty}\int_{\R^3\setminus K}\rho^{a_k}(s,t,x)\p_t\varphi(t,x)dxdt+\int_{\R^3\setminus K}\rho^{a_k}_0(x)\varphi(s,x)\,dx=0$ is used. 
It is clear that  
\begin{align*}
&\int_s^{\infty}\int_K \rho^{a_k}(s,t,x)\p_t\varphi(t,x)dxdt\to \int_s^{\infty}\int_K \rho(s,t,x)\p_t\varphi(t,x)dxdt\quad\mbox{ as $k\to\infty$},\\
&\int_K\rho^{a_k}_0(x)\varphi(s,x)\,dx\to \int_K\rho_0(x)\varphi(s,x)\,dx\quad\mbox{ as $k\to\infty$}.
\end{align*}
Let $T>0$ be such that supp$(\varphi)\subset[s-T,s+T]\times[-T,T]^3$. 
For each (sufficiently small) $\ep>0$, we find $k_\ep\in\N$ such that 
$$\norm v^{a_{k}}-v\norm_{L^1([s,s+T]\times K)^3}<\ep,\quad \norm \nabla\cdot v^{a_{k}}-\nabla\cdot v\norm_{L^1([s,s+T]\times K)}<\ep\quad \mbox{ for all $k\ge k_\ep$.}$$
Hence, it holds that 
\begin{align*}
&\int_s^{\infty}\int_Kv^{a_k}(t,x)\rho^{a_k}(s,t,x)\cdot\nabla\varphi(t,x)dxdt
-\int_s^{\infty}\int_K v(t,x)\rho(s,t,x)\cdot\nabla\varphi(t,x)dxdt\\
&= \uwave{\int_s^{s+T}\int_K (v^{a_k}(t,x) - v^{a_{k_\ep}}(t,x) )\rho^{a_k}(s,t,x)\cdot\nabla\varphi(t,x)dxdt}_{\rm(i)} \\
&\quad + \uwave{\int_s^{s+T}\int_K v^{a_{k_\ep}}(t,x) (\rho^{a_k}(s,t,x)-\rho(s,t,x) )\cdot\nabla\varphi(t,x)dxdt }_{\rm(ii)}\\
&\quad +  \uwave{\int_s^{s+T}\int_K (v^{a_k{_\ep}}(t,x) - v(t,x) )\rho(s,t,x)\cdot\nabla\varphi(t,x)dxdt }_{\rm(iii)}, \\
&|{\rm(i)}|\le \max|\nabla\varphi |\beta \norm v^{a_k}-v^{a_{k_\ep}}\norm_{L^1([s,s+T]\times K)^3}\le 2\max|\nabla\varphi |\beta \ep\quad \mbox{for all $k\ge\ k_\ep$},\\
&|{\rm(iii)}|\le \max|\nabla\varphi |\beta \norm v^{{a_{k_\ep}}}-v\norm_{L^1([s,s+T]\times K)^3}\le \max|\nabla\varphi |\beta \ep,\\
&{\rm (ii)}\to 0\quad\mbox{ as $k\to\infty$}, 
\end{align*}
where we note that  $v^{a_{k_\ep}}\in L^2([s,s+T]\times K)^3$ in (ii). 
This confirms that 
$$\int_s^{\infty}\int_Kv^{a_k}(t,x)\rho^{a_k}(s,t,x)\cdot\nabla\varphi(t,x)dxdt
\to\int_s^{\infty}\int_K v(t,x)\rho(s,t,x)\cdot\nabla\varphi(t,x)dxdt\quad \mbox{as $k\to\infty$}.$$
A similar argument yields 
$$\int_s^{\infty}\int_K  (\nabla\cdot v^{a_k}(t,x))\rho^{a_k}(s,t,x)\varphi(t,x)dxdt
\to\int_s^{\infty}\int_K  (\nabla\cdot v(t,x))\rho(s,t,x)\varphi(t,x)dxdt\quad \mbox{as $k\to \infty$}.$$
Therefore, we obtain    
\begin{align}\label{weak21}
0=&\int_s^{\infty}\int_K \Big\{\rho(s,t,x)\p_t\varphi(t,x)+v(t,x)\rho(s,t,x)\cdot\nabla\varphi(t,x)\\\nonumber
&\qquad + (\nabla\cdot v(t,x))\rho(s,t,x)\varphi(t,x)\Big\}dxdt+\int_K\rho_0(x)\varphi(s,x)\,dx.
\end{align}
Since $v\equiv 0$ in $\R\times(\R^3\setminus\Omega)$, it holds that   
\begin{align}\label{yo}
&0=\int_s^{\infty}\int_\Omega \Big\{\rho(s,t,x)\p_t\varphi(t,x)+v(t,x)\rho(s,t,x)\cdot\nabla\varphi(t,x)\\\nonumber
&\quad + (\nabla\cdot v(t,x))\rho(s,t,x)\varphi(t,x)\Big\}dxdt+\int_s^{\infty}\int_{K\setminus\overline{\Omega}} \rho(s,t,x)\p_t\varphi(t,x)\,dxdt+\int_K\rho_0(x)\varphi(s,x)\,dx.
\end{align}
If $\varphi$ is such that supp$(\varphi)\subset(s,\infty)\times(K\setminus \overline{\Omega})$, \eqref{yo} yields
$$0=\int_s^{\infty}\int_{K\setminus\overline{\Omega}} \rho(s,t,x)\p_t\varphi(t,x)\,dxdt,$$
which implies that $\rho(s,t,\cdot)$ is independent of $t\ge s$ in $K\setminus\overline{\Omega}$. Hence, if $\varphi$ is such that supp$(\varphi)\subset\R\times(K\setminus \overline{\Omega})$,  \eqref{yo} yields
$$0=-\int_{K\setminus\overline{\Omega}} \rho(s,t,x)\varphi(s,x)\,dx+\int_{K\setminus\overline{\Omega}}\rho_0(x)\varphi(s,x)\,dx
\quad \mbox{ for all $t\in[s,\infty)$},$$
which implies that $\rho(s,t,\cdot)=\rho_0$ on $K\setminus \overline{\Omega}$ for all $t\ge s$.  
Thus we see that \eqref{weak21} is equivalent to \eqref{LT-weak}. 
A similar argument confirms \eqref{LT-weak2} to conclude that $\rho(s,\cdot,\cdot)|_{x\in \Omega}$ is a weak solution of \eqref{LT}.
\end{proof}
\indent We investigate evolution of the norm of weak solutions of \eqref{LT}, which leads to the uniqueness of \eqref{LT}, continuity of weak solution in the $t$-variable,  strong convergence of $\{\rho^k(s,\cdot,\cdot)\}$, etc. 
For this purpose, we need a technical lemma, which is the key technical part of DiPerna-Lions theory, based on mollification by the mollifier $\{\eta^\ep\}_{\ep>0}$, $\eta^\ep:\R^3\to\R$ given as 
\begin{align*}
\eta^\ep(x):=\frac{1}{\ep^3}\eta\Big(\frac{x}{\ep}\Big),\quad 
\eta(x):=
\begin{cases}
Ce^{\frac{1}{|x|^2-1}}\mbox{\quad if $|x|<1$}\\
0\mbox{\quad if $|x|\ge1$},
\end{cases}\mbox{ with $\dis\int_{{\R^3}}\eta(x)dx=1$},
\end{align*}
where $\ep>0$ is taken small enough so that 
\begin{align}\label{ne1}
\{x-y\,|\, y\in B_\ep (0)  \}\cap \Omega =\emptyset\mbox{\quad for all $x\in \R^3\setminus K$}. 
\end{align}
Let $\rho(s,\cdot,\cdot)\in L^\infty(\R\times\Omega)$ be a weak solution of \eqref{LT}.
We introduce $\rho_\pm(s,\cdot,\cdot)$ and  $v_\pm$ as
\begin{align*}
\rho_+(s,t,x)&:=
\begin{cases}
\rho(s,t,x)&\quad\mbox{if $t\ge s$, $x\in\Omega$},\\
\rho_0(x)&\quad\mbox{if $t<s$, $x\in\Omega$},\\
0&\quad\mbox{if $x\not\in\Omega$},
\end{cases}\qquad 
\rho_-(s,t,x):=
\begin{cases}
\rho_0(x)&\quad\mbox{if $t>s$, $x\in\Omega$},\\
\rho(s,t,x)&\quad\mbox{if $ t\le s$, $x\in\Omega$},\\
0&\quad\mbox{if $x\not\in\Omega$}.
\end{cases}\\
v_+(t,x)&:=
\begin{cases}
v(t,x)&\quad\mbox{if $t\ge s$},\\
0&\quad\mbox{if $t<s$},\\
\end{cases}\qquad\qquad \qquad\quad\,
v_-(t,x):=
\begin{cases}
0&\quad\mbox{if $t>s$},\\
v(t,x)&\quad\mbox{if $t\le s$}, 
\end{cases}
\end{align*}
where $v$ is extended to be zero outside $\Omega$.
Set $\rho_\pm^\ep(t,\cdot):=\eta^\ep\ast \rho_\pm(s,t,\cdot)$, $\rho_0^\ep:=\eta^\ep\ast \rho_0$ with $\rho_0$ extended to be zero outside $\Omega$\footnote{We use the  $0$-extension of $\rho_0$ here, no matter which way (I1) or (I2) is used to construct $\rho(s,\cdot,\cdot)$.}. 
Note that $\rho^\ep_+:\R\times\R^3\to \R$, $\rho^\ep_-:{\R}\times\R^3\to \R$ are  measurable\footnote{$F^\ep(s,t,x,y):=\eta^\ep(x-y)\rho_\pm(s,t,y)$ is measurable and  Fubini's theorem implies that $\rho^\ep_\pm(s,t,x)=\int_{\R^3}F^\ep(s,t,x,y)dy$ is measurable in $(t,x)$.}.  
In the upcoming investigation of the norm of weak solutions, we require absolute continuity of $\norm \rho(s,t,\cdot)\norm_{L^2(\Omega)}$ with respect to $t$  and the (a priori undefinable) value of $\norm \rho(s,s,\cdot)\norm_{L^2(\Omega)}$.  
In the following Lemma  \ref{LT-commu}, the range of the time integration contains $[s-1,s]$ or $[s,s+1]$, on which $\rho^\ep_\pm(s,t,\cdot)=\rho_0^\ep$. This leads to justification of  ``$\norm \rho(s,s,\cdot)\norm_{L^2(\Omega)}=\norm \rho^0\norm_{L^2(\Omega)}$''. 
\begin{Lemma}[key]\label{LT-commu}
1. Let $s\in\R$ be arbitrary. For each $\ep>0$, there exists $R_+^\ep(s,\cdot,\cdot)\in L^1_{\rm loc}([s-1,\infty)\times{K} )$ such that for all $\varphi\in C^\infty_0(\R^4)$ we have    
\begin{align*}
&\int_{s-1}^{\infty}\int_K\Big\{ \rho_+^\ep(s,t,x)\p_t\varphi(t,x)+v_+(t,x)\rho_+^\ep(s,t,x)\cdot\nabla\varphi(t,x)\\
&+ (\nabla\cdot v_+(t,x))\rho^\ep_+(s,t,x)\varphi(t,x)\Big\}dxdt
+\int_{K}  \rho_0^\ep(x)\varphi(s-1,x)\,dx=\int_{s-1}^{{\infty}}\int_K R_+^\ep(s,t,x)\varphi(t,x)dxdt.
\end{align*}
Furthermore, for each $T>0$, it holds that  $\norm R_+^\ep(s,\cdot,\cdot)\norm_{L^1([s-1,s+T]\times{K} )}\to0$  as $\ep\to0$.

2.  Let $s\in\R$ be arbitrary. For each $\ep>0$, there exists $R_-^\ep(s,\cdot,\cdot)\in L^1_{\rm loc}((-\infty,s+1]\times{K} )$ such that for all $\varphi\in C^\infty_0(\R^4)$ we have     
\begin{align*}
&\int_{-\infty}^{s+1}\int_{K}  \Big\{ \rho_-^\ep(s,t,x)\p_t\varphi(t,x)+v_-(t,x)\rho_-^\ep(s,t,x)\cdot\nabla\varphi(t,x)\\
& +   (\nabla\cdot v_-(t,x))\rho^\ep_-(s,t,x)\varphi(t,x) \Big\}dxdt
-\int_{K} \rho_0^\ep(x)\varphi(s+1,x)\,dx=\int_{-\infty}^{s+1}\int_{K}  R_-^\ep(s,t,x)\varphi(t,x)dxdt.
\end{align*}
Furthermore, for each $T>0$,  it holds that  $\norm R_-^\ep(s,\cdot,\cdot)\norm_{L^1([s-T,s+1]\times{K} )}\to0$  as $\ep\to0$.
\end{Lemma}

\begin{proof}
Take any test function  $\varphi\in C^\infty_0(\R^4)$.
Since $v_+=0$ and $\rho^\ep_+(s,\cdot,\cdot)=\rho_0^\ep$ for all $t<s$, we have 
\begin{align*}
I:=&\int_{s-1}^{\infty}\int_{K}  \Big\{\rho_+^\ep(s,t,x)\p_t\varphi(t,x)+v_+(t,x)\rho_+^\ep(s,t,x)\cdot\nabla\varphi(t,x)+ (\nabla\cdot v_+(t,x))\rho^\ep_+(s,t,x)\varphi(t,x)\Big\}dxdt\\
=&\int_{B_\ep(0)}\Big[\int_{s}^{\infty}\int_{K}  \Big\{\rho_+(s,t,x-y)\p_t\Big(\eta^\ep(y)\varphi(t,x)\Big)+v(t,x)\rho_+(s,t,x-y)\cdot\nabla_x\Big(\eta^\ep(y)\varphi(t,x)\Big)\\
&+(\nabla_x\cdot v(t,x)) \rho_+(s,t,x-y) \Big( \eta^\ep(y)\varphi(t,x) \Big)   \Big\}dxdt
\Big]dy
+\int_{K}   \rho_0^\ep(x) \varphi(s,x)\,dx-\int_{K}   \rho_0^\ep(x)\varphi(s-1,x)\,dx,
\end{align*} 
where $\int_Kdx$ inside $[\,\,\,]$ can be $\int_{\R^3}dx$, because we have $\rho_+(s,t,x-y)=0$ for all $y\in B_\ep(0)$ and $x\in \R^3\setminus K$ due to the choice of $\ep>0$ in \eqref{ne1}.  
By change of the variable $x-y$ into $z$ in $\int_{\R^3}dx$, we have 
\begin{align*}
I=&\uwave{\int_{B_\ep(0)}\Big[\int_{s}^{\infty}\int_{\R^3} \Big\{\rho_+(s,t,z)\p_t\Big(\eta^\ep(y)\varphi(t,y+z)\Big)+v(t,z)\rho_+(s,t,z)\cdot\nabla_z\Big(\eta^\ep(y)\varphi(t,y+z)\Big)}\\
&\uwave{+(\nabla_z\cdot v(t,z)) \rho_+(s,t,z) \Big( \eta^\ep(y)\varphi(t,y+z) \Big)   \Big\}dzdt
\Big]dy}_{\rm (i)}\\
&+ \uwave{\int_{B_\ep(0)}\Big[
\int_{s}^{\infty}\int_{\R^3}\Big\{\big(v(t,y+z)-v(t,z)\big)\rho_+(s,t,z)\cdot\nabla_z\Big(\eta^\ep(y)\varphi(t,y+z)\Big) }\\
&\uwave{+    \big(\nabla_z\cdot v(t,y+z)-\nabla_z\cdot v(t,z)\big) \rho_+(s,t,z) \Big( \eta^\ep(y)\varphi(t,y+z) \Big)    \Big\}dzdt
\Big]dy}_{\rm (ii)}\\
&+\int_{K}   \rho_0^\ep(x) \varphi(s,x)\,dx-\int_{K}   \rho_0^\ep(x)\varphi(s-1,x)\,dx.
\end{align*}
Since $\rho_+(s,\cdot,\cdot)=0$ outside $\Omega$, the integral $\int_{\R^3}dx$ in (i) can be replaced with $\int_{\Omega}dx$; we have $\rho_+(s,t,z)=\rho(s,t,z)$ for $(t,z)\in[s,\infty)\times{\Omega}$ satisfying \eqref{LT} weakly;   
for each fixed $y\in B_\ep(0)$, $\tilde{\varphi}(t,z;y):=\eta^\ep(y)\varphi(t,y+z)$ can be an admissible test function in \eqref{LT-weak}\footnote{This is the reason why  we do not take test functions  as $\varphi\in C^\infty([s,\infty)\times{\Omega} )$ with supp$(\varphi)\subset[s,\infty)\times{\Omega} $ compact  in \eqref{LT-weak}.}. 
Hence, we obtain 
\begin{align*}
{\rm (i)}=& -\int_{B_\ep(0)} \int_{\Omega}  \rho_0(z)\tilde{\varphi}(s,z;y)dzdy= -\int_{B_\ep(0)} \int_{\R^3}  \rho_0(z)\tilde{\varphi}(s,z;y)dzdy,
\end{align*}  
where we recall that $\rho_0\equiv 0$ in $\R^3\setminus\Omega$.
By  \eqref{ne1}, we see that $\rho^\ep_0\equiv 0$ in $\R^3\setminus K$ and 
\begin{align*}
{\rm (i)}
=& -\int_{B_\ep(0)} \int_{\R^3} \rho_0(x-y)\eta^\ep(y)\varphi(s,x)dxdy
= -\int_{\R^3} \rho_0^\ep(x)\varphi(s,x)dx
=-\int_{{K} } \rho_0^\ep(x)\varphi(s,x)dx.
\end{align*}
By change of the variable $y+z$ into $x$,  the term (ii) becomes 
\begin{align*}
{\rm (ii)}=&\int_{B_\ep(0)}\Big[
\int_{s}^{\infty}\int_{\R^3}\Big\{\big(v(t,x)-v(t,x-y)\big)\rho_+(s,t,x-y)\cdot\nabla_x\Big(\eta^\ep(y)\varphi(t,x)\Big) \\
&+ \big(\nabla_x\cdot v(t,x)-\nabla_x\cdot v(t,x-y)\big) \rho_+(s,t,x-y) \Big( \eta^\ep(y)\varphi(t,x) \Big)    \Big\}dxdt
\Big]dy\\
=&\int_{s}^{\infty}\int_{\R^3}\Big\{ \Big(\int_{B_\ep(0)}\rho_+(s,t,x-y) \eta^\ep(y)\,dy\Big) v(t,x)  \cdot\nabla_x\varphi(t,x) \\
& -  \Big(\int_{B_\ep(0)} \rho_+(s,t,x-y) v(t,x-y)\eta^\ep(y)\,dy\Big)  \cdot\nabla_x\varphi(t,x) \Big\}dxdt  \\
&+ \int_{s}^{\infty}\int_{\R^3}\Big\{ \Big(\int_{B_\ep(0)}   \rho_+(s,t,x-y)  \eta^\ep(y)\,dy\Big) (\nabla_x\cdot v(t,x))\varphi(t,x)\\
& -   \Big(\int_{B_\ep(0)}  \rho_+(s,t,x-y)(\nabla_x\cdot v(t,x-y) ) \eta^\ep(y)\,dy \Big)\varphi(t,x)   \Big\}dxdt\\
=&\int_{s}^{\infty}\int_{\R^3}\Big[ 
\rho^\ep_+(s,t,x) v(t,x) \cdot \nabla_x \varphi(t,x) - \{\eta^\ep\ast (\rho_+(s,t,\cdot)v(t,\cdot))\}(x) \cdot \nabla_x \varphi(t,x)
\Big]dxdt \\
&+ \int_{s}^{\infty}\int_{\R^3}\Big[ \rho_+^\ep(s,t,x)  (\nabla_x\cdot v(t,x))\varphi(t,x) -  \{ \eta^\ep \ast (\rho_+(s,t,\cdot)(\nabla_x\cdot v(t,\cdot)))\}(x) \varphi(t,x)   \Big]dxdt.
\end{align*}
Applying integration by parts to the first integral on the right hand side, we obtain 
\begin{align*}
{\rm (ii)}
=&\int_{s}^{\infty}\int_{\R^3}\Big[ 
-\nabla_x \rho^\ep_+(s,t,x) \cdot v(t,x)  \varphi(t,x) + \nabla_x\cdot \{\eta^\ep\ast (\rho_+(s,t,\cdot)v(t,\cdot))\}(x)\varphi(t,x)
\\
&-  \{ \eta^\ep \ast (\rho_+(s,t,\cdot)(\nabla_x\cdot v(t,\cdot)))\}(x) \varphi(t,x)   \Big]dxdt.
\end{align*}
Due to  $v_+\equiv0$ in $(-\infty,s)\times(\R^3\setminus \Omega)$ and \eqref{ne1}, we have 
\begin{align*}
{\rm (ii)}
=&\int_{s-1}^{\infty}\int_{K}\Big[ 
-\nabla_x \rho^\ep_+(s,t,x) \cdot v_+(t,x)  \varphi(t,x) + \nabla_x\cdot \{\eta^\ep\ast (\rho_+(s,t,\cdot)v_+(t,\cdot))\}(x)\varphi(t,x)
\\
&-  \{ \eta^\ep \ast (\rho_+(s,t,\cdot)(\nabla_x\cdot v_+(t,\cdot)))\}(x) \varphi(t,x)   \Big]dxdt. 
\end{align*}
Therefore, we obtain the first assertion with  
\begin{align*}
R_+^\ep(s,t,x):=&  -\nabla_x \rho^\ep_+(s,t,x) \cdot v_+(t,x)  + \nabla_x\cdot \{\eta^\ep\ast (\rho_+(s,t,\cdot)v_+(t,\cdot))\}(x)\\
&-  \{ \eta^\ep \ast (\rho_+(s,t,\cdot)(\nabla\cdot v_+(t,\cdot)))\}(x).
\end{align*}
%
\indent Next, we show that $\norm R^\ep_+(s,\cdot,\cdot)\norm_{L^1([s-1,s+T]\times{K} )}=\norm R^\ep_+(s,\cdot,\cdot)\norm_{L^1([s,s+T]\times{K} )}\to0$  as $\ep\to0$ for each $T>0$. 
Since 
\begin{align*}
&R_+^\ep(s,t,x)=  \uwave{-\nabla_x \rho^\ep_+(s,t,x) \cdot v_+(t,x)  + \nabla_x\cdot \{\eta^\ep\ast (\rho_+(s,t,\cdot)v_+(t,\cdot))\}(x)
-\rho_+(s,t,x)(\nabla\cdot v_+(t,x))}_{\rm (a)}\\
&\quad +\uwave{\rho_+(s,t,x)(\nabla\cdot v_+(t,x)) -  \{ \eta^\ep \ast (\rho_+(s,t,\cdot)(\nabla\cdot v_+(t,\cdot)))\}(x)}_{\rm(b)}
\end{align*}
and $\norm {\rm (b)}\norm_{L^1([s,s+T]\times{K} )}\to0$ as $\ep\to0$, we obtain the assertion by showing $\norm {\rm (a)}\norm_{L^1([s,s+T]\times{K} )}\to0$ as $\ep\to0$.
For a.e. $t\in [s,s+T]$, {where} $v_+=v$, we have 
\begin{align*}
& \!\!\!\!\!-\nabla_x \rho^\ep_+(s,t,x) \cdot v_+(t,x)  + \nabla_x\cdot \{\eta^\ep\ast (\rho_+(s,t,\cdot)v_+(t,\cdot))\}(x)\\
&= \int_{B_\ep(x)} \Big\{ -v(t,x) \rho_+(s,t,y)\cdot\nabla \eta^\ep(x-y) +\rho_+(s,t,y)v(t,y)\cdot\nabla\eta^\ep(x-y)   \Big\} dy\\
&= -\sum_{i=1}^3\int_{B_\ep(0)} \rho_+(s,t,x-y)\{ v_i(t,x) -v_i(t,x-y)\}\p_{x_i}\eta^\ep(y)  dy.
\end{align*}
It follows from the mollification $\{v_i^\nu(t,\cdot)\}_{\nu>0}$ of $v_i(t,\cdot)\in W^{1,1}(\R^3)$ with supp$(v_i(t,\cdot))\subset\overline{\Omega}$ that for any $\psi\in C^\infty_0(K;\R)$,
\begin{align*}
\int_K\{ v_i^\nu(t,x) -v_i^\nu(t,x-y)\}\psi(x)dx=\int_K\int_0^1\nabla v_i^\nu(t,x-hy)\cdot ydh\psi(x)dx;
\end{align*}
sending $\nu\to0+$, we get 
\begin{align*}
\int_K\{ v_i(t,x) -v_i(t,x-y)\}\psi(x)dx=\int_K\int_0^1\nabla v_i(t,x-hy)\cdot ydh\psi(x)dx
\end{align*}
to confirm that 
$$v_i(t,x) -v_i(t,x-y)=\sum_{{j=1}}^3\int_0^1\p_{x_j}v_i(t,x-hy)y_jdh\mbox{ \quad a.e. $x\in\R^3$}.$$
Since $\int_{B_\ep(0)} y_j \p_{x_i}\eta^\ep(y)\,dy=-\delta_{ij}$ (Kronecker's $\delta$), we find that 
\begin{align*}
 \rho_+(s,t,x)(\nabla\cdot v(t,x))  =-\sum_{i,j=1}^3\int_{B_\ep({0})} \rho_+(s,t,x)\p_{x_j} v_i(t,x)y_j \p_{x_i}\eta^\ep(y)  dy.
\end{align*}
Hence, it holds that
\begin{align*}
&\norm {\rm (a)}\norm_{L^1([s,s+T]\times{K} )}
\le\int_s^{s+T}\int_K \int_{B_\ep(0)} 
\Big| \sum_{i,j=1}^3\rho_+(s,t,x)\p_{x_j} v_i(t,x)y_j \p_{x_i}\eta^\ep(y) \\
&\quad\qquad\qquad\qquad\qquad\qquad\qquad -\sum_{i=1}^3 \rho_+(s,t,x-y)\{ v_i(t,x) -v_i(t,x-y)\}\p_{x_i}\eta^\ep(y)  
\Big|dydxdt\\
&= \int_{B_\ep(0)}\int_s^{s+T}\int_K \Big| \sum_{i,j=1}^3\rho_+(s,t,x)\p_{x_j} v_i(t,x)y_j \p_{x_i}\eta^\ep(y) \\
&\quad\qquad\qquad\qquad\qquad\qquad\qquad -\sum_{i=1}^3 \rho_+(s,t,x-y)\{ v_i(t,x) -v_i(t,x-y)\}\p_{x_i}\eta^\ep(y)  
\Big|dydxdt\\
&= \int_{B_\ep(0)}\int_s^{s+T}\int_K \Big| \sum_{i,j=1}^3\rho_+(s,t,x)\p_{x_j} v_i(t,x)y_j \p_{x_i}\eta^\ep(y) \\
&\quad\qquad\qquad\qquad\qquad\qquad\qquad - \rho_+(s,t,x-y)\sum_{i,j=1}^3\int_0^1\p_{x_j}v_i(t,x-hy)y_jdh \p_{x_i}\eta^\ep(y)  
\Big|dydxdt\\
&\le   \sum_{i,j=1}^3 \int_{B_\ep(0)}\int_s^{s+T}\int_K\int_0^1\Big|  \{ \rho_+(s,t,x)\p_{x_j} v_i(t,x)-\rho_+(s,t,x-y)\p_{x_j}v_i(t,x-hy)\}y_j \p_{x_i}\eta^\ep(y)\Big|dhdxdtdy\\
&=  \sum_{i,j=1}^3\int_0^1 \int_{B_\ep(0)}\int_s^{s+T}\int_K\Big|  \{ \rho_+(s,t,x)\p_{x_j} v_i(t,x)-\rho_+(s,t,x-y)\p_{x_j}v_i(t,x-hy)\}y_j \p_{x_i}\eta^\ep(y)\Big|dhdxdtdy.
\end{align*}
Since $|y_j|<\ep$ and  $\int_{B_\ep(0)}    \ep|\nabla\eta^\ep(y)|dy$ is bounded by some constant $\tilde{C}$ independent of $\ep>0$, it holds that 
\begin{align*}
\norm {\rm (a)}\norm_{L^1([s,s+T]\times{K} )}
&\le \sum_{i,j=1}^3 \tilde{C} \sup_{|y|,|y'|<\ep} \norm \rho_+(s,\cdot,\cdot)\p_{x_j} v_i(\cdot,\cdot)-\rho_+(s,\cdot,\cdot-y)\p_{x_j}v_i(\cdot,\cdot-y')\norm_{L^1([s,s+T]\times\R^3)}\\
&\to 0\quad\mbox{ as $\ep\to0$}.
\end{align*}
A similar argument provides 2. 
\end{proof}
As an important consequence of Lemma \ref{LT-commu}, we state evolution of the $L^p$-norm ($1\le p <\infty$) of weak solutions, where we only deal with the case $p=2$. 
\begin{Lemma}\label{LT-norm}
Let $s\in\R$ be arbitrary and $\rho(s,\cdot,\cdot)\in L^\infty(\R\times\Omega)$ be a weak solution of \eqref{LT}. Then, after possible change of the value  of $\norm\rho(s,t,\cdot)\norm_{L^2(\Omega)}$ on a null set of $\R$, it holds that   
\begin{align}
\label{LT-norm2}
\norm \rho(s,t,\cdot)\norm_{L^2(\Omega)}^2
=&\begin{cases}\medskip
&\dis\norm \rho_0\norm_{L^2(\Omega)}^2+\int_s^t \int_\Omega (\nabla\cdot v(r,x))\rho(s,r,x)^2dxdr    \quad \mbox{ for all $t\ge s$,} \\
&\dis\norm \rho_0\norm_{L^2(\Omega)}^2-\int^s_t \int_\Omega (\nabla\cdot v(r,x))\rho(s,r,x)^2dxdr    \quad \mbox{ for all $t\le s$.} 
\end{cases}
\end{align}
\end{Lemma}

\begin{proof}
It follows from  the assertion 1 of Lemma \ref{LT-commu} that for each $T>0$ and each $\varphi\in C^\infty([s-1,s+T]\times\Omega)$ with supp$(\varphi)\subset(s-1,s+T)\times\Omega$ we have   
\begin{align}\label{giro1}
\int_{s-1}^{s+T}\int_\Omega \rho_+^\ep(s,t,x)\p_t\varphi(t,x){\,dxdt} = \int_{s-1}^{s+T}\int_\Omega \{v_+(t,x)\cdot \nabla \rho_+^\ep(s,t,x) +R_+^\ep(s,t,x)\}\varphi(t,x) \,dxdt,
\end{align}
where we note that  now $\nabla \rho_+^\ep$ makes sense. 
Hence, $\rho_+^\ep(s,\cdot,\cdot)$ is weakly $t$-differentiable with 
\begin{align}\label{kae}
\p_t\rho_+^\ep(s,\cdot,\cdot)=-v_+\cdot\nabla\rho_+^\ep(s,\cdot,\cdot)-R_+^\ep(s,\cdot,\cdot)\in L^1([s-1,s+T]\times\Omega).
\end{align}
Then, for a.e. $x\in\Omega$, the function $\rho_+^\ep(s,\cdot,x)$ is differentiable for a.e. $t\in [s-1,s+T]$ in the classical sense; the (classical)  derivative $\frac{\p}{\p t}\rho_+^\ep(s,t,x)$ exists for a.e. $(t,x)\in[s-1,s+T]\times\Omega$ and coincides with the weak $t$-derivative. 
Therefore, we obtain 
\begin{align}\label{aki1}
\p_t\{ \rho_+^\ep(s,t,x)^2\} +v_+(t,x)\cdot\nabla\{\rho_+^\ep(s,t,x)^2\}+&2R_+^\ep(s,t,x)\rho_+^\ep(s,t,x)=0\\\nonumber
&\qquad \mbox{a.e. in $[s-1,s+T]\times\Omega$}.
\end{align}
For each $\varphi\in C^\infty([s-1,s+T]\times\Omega)$ with supp$(\varphi)\subset(s-1,s+T)\times\Omega$, we have   
\begin{align*}
\int_{s-1}^{s+T}\int_\Omega\Big[ \rho_+^\ep(s,t,x)^2 \p_t\varphi(t,x) &- v_+(t,x)\cdot\nabla\{\rho_+^\ep(s,t,x)^2\} \varphi(t,x)\\\nonumber
&-2R_+^\ep(s,t,x)\rho_+^\ep(s,t,x)\varphi(t,x)\Big]dxdt=0. 
\end{align*}
Let $\{\chi^k\}_{k\in\N}\subset C^\infty_0(\Omega)$ be a smooth approximation of the indicator function $\chi_\Omega$ of $\Omega$.  
 Taking $\varphi(t,x)=\chi^k(x)f(t)$ with $f\in C^\infty([s-1,s+T])$ with supp$(f)\subset(s-1,s+T)$, we have 
 \begin{align*}
 \int_{s-1}^{s+T}\int_\Omega \rho_+^\ep(s,t,x)^2 \chi^k (x) f'(t) &- v_+(t,x)\cdot\nabla\{\rho_+^\ep(s,t,x)^2\} \chi^k(x)f(t)\\
&\qquad  -2R_+^\ep(s,t,x)\rho_+^\ep(s,t,x)\chi^k(x)f(t)\,dxdt=0. 
 \end{align*}
 Since $\{\chi^k\}_{k\in\N}$ converges to $\chi_\Omega$ a.e. pointwise, Lebesgue's dominated convergence theorem with $k\to\infty$  yields   
 \begin{align}\label{kae2}
 \int_{s-1}^{s+T}\int_\Omega\Big[ \rho_+^\ep(s,t,x)^2f'(t) &+ (\nabla\cdot v_+(t,x))\{\rho_+^\ep(s,t,x)^2\}f(t)
  -2R_+^\ep(s,t,x)\rho_+^\ep(s,t,x)f(t)\Big]dxdt=0. 
 \end{align}
 since $ \int_\Omega v_+(t,x)\cdot\nabla\{\rho_+^\ep(s,t,x)^2\}dx
 = -\int_\Omega (\nabla\cdot v_+(t,x))\rho_+^\ep(s,t,x)^2dx$ for a.e. $t\in[s-1,s+T]$ and $\rho_+^\ep(s,\cdot,\cdot)\to\rho_+(s,\cdot,\cdot)$ as $\ep\to 0$ a.e. pointwise, as well as $R_+^\ep(s,\cdot,\cdot)\to 0$  as $\ep\to 0$ in $L^1([s-1,s+T]\times\Omega)$, we obtain 
 \begin{align}\label{beta1}
 \int_{s-1}^{s+T}\int_\Omega \rho_+(s,t,x)^2f'(t)dxdt= - \int_{s-1}^{s+T}\Big\{\int_\Omega (\nabla\cdot v_+(t,x))\rho_+(s,t,x)^2dx\Big\} f(t)dt.  
 \end{align} 
 Since $f$ is arbitrary, $\norm \rho_+(s,t,\cdot)\norm_{L^2(\Omega)}^2$ is weakly $t$-differentiable.
 Hence $\norm \rho_+(s,t,\cdot)\norm_{L^2(\Omega)}^2$  is absolutely continuous  on $[s-1,s+T]$ (after possible change of value on a null set).  
 Then, we find that for all $t_0\in[s-1,s)$ and $t\in(s, s+T]$,
 \begin{align*}
 \norm \rho_+(s,t,\cdot)\norm_{L^2(\Omega)}^2 &=\norm \rho_+(s,t_0,\cdot)\norm_{L^2(\Omega)}^2+\int_{t_0}^{t}\int_\Omega (\nabla\cdot v_+(r,x))\rho_+(s,r,x)^2dxdr.
 \end{align*}
Since $\rho_+(s,t,\cdot)=\rho_0$ and $v_+(t,\cdot)=0$ for all $t<s$, we obtain  
 \begin{align*}
 \norm \rho(s,t,\cdot)\norm_{L^2(\Omega)}^2 &=\norm \rho_0\norm_{L^2(\Omega)}^2+\int_{s}^{t}\int_\Omega (\nabla\cdot v(r,x))\rho(s,r,x)^2dxdr.
 \end{align*}
 Since $T>0$ is arbitrary, we conclude the assertion for $t\ge s$.  
 The case $ t\le s$ can be treated in the same way.
\end{proof}
\noindent Gronwall's inequality and \eqref{LT-norm2} yield the uniqueness of \eqref{LT}. 
It is remarkable that $\rho^\ep_\pm$ has the  $C^0_tL^2_x$-regularity forming a Cauchy sequence there. 
The uniqueness implies that the weak solution of \eqref{LT}   has the  $C^0_tL^2_x$-regularity.  We remark that the assumption $\nabla \cdot v\in L^1_{\rm loc}(\R;L^\infty(\Omega))$ is crucial to apply Gronwall's inequality to  \eqref{LT-norm2}. 
\begin{Prop}\label{LT-unique}
For each $s\in\R$ and $\rho_0\in L^\infty(\Omega)$, a weak solution $\rho(s,\cdot,\cdot)$ of  \eqref{LT} is unique in $L^\infty(\R\times\Omega)$.  
Furthermore, after possible change of the values of $\rho(s,\cdot,\cdot)$ on a null set of $\R\times\Omega$, $\rho(s,\cdot,\cdot)$ belongs to $C^0(\R;L^2(\Omega))$ with $\rho(s,s,\cdot)=\rho_0$\footnote{From now on,  we mean by a weak solution $\rho(s,\cdot,\cdot)$ of \eqref{LT} this  representative. Hence, $\rho(s,t,\cdot)$ has a value in $L^2(\Omega)$ for every $t\in\R$.}.
\end{Prop}

\begin{proof}
Let $\rho(s,\cdot,\cdot),\tilde{\rho}(s,\cdot,\cdot)\in  L^\infty(\R\times\Omega)$ be two weak solutions of \eqref{LT}. Then, $u(s,\cdot,\cdot):=\rho(s,\cdot,\cdot)-\tilde{\rho}(s,\cdot,\cdot)\in  L^\infty(\R\times\Omega)$ is a weak solution of  \eqref{LT}$|_{\rho_0=0}$. 
Due to  $L_{\rm loc}(\R;L^\infty(\Omega))$-regularity of $\nabla\cdot v$, Gronwall's inequality applied to the equality \eqref{LT-norm2}$|_{\rho=u}$ confirms that $u=0$ and the uniqueness is proven.   

Let $T>0$ be arbitrary. Due to \eqref{kae} and $\rho_+^\ep(s,t,\cdot)=\rho_0^\ep$ for all $t<s$, it holds that for a.e. $(t,x)\in[s-1,s+T]\times\Omega$, 
$$\rho_+^\ep(s,t,x)=\rho_0^\ep(x)+\int_{s}^t \{-v_+(r,x)\cdot\nabla\rho_+^\ep(s,r,x)-R_+^\ep(s,r,x)\}dr=:f^\ep(t,x),$$
where $f^\ep$ belongs to $C^0(\R;L^1(\Omega))$. 
Hence, there exists a null set $N\subset[s-1,s+T]$ such that $\rho_+^\ep(s,t,\cdot)$ is equal to $f^\ep(t,\cdot)$ for all $t\in[s-1,s+T]\setminus N$.  
Changing the value of $\rho_+^\ep(s,\cdot,\cdot)$ to be $f^\ep$ on the null set $N\times\Omega$,  we identify $\rho_+^\ep(s,\cdot,\cdot)$ as an element of $C^0(\R;L^1(\Omega))$.  
Then, we see that $\rho_+^\ep(s,\cdot,\cdot)$ actually belongs to $C^0(\R;L^2(\Omega))$, because as $h\to 0$,  
\begin{align*}
\norm \rho_+^\ep(s,t+h,\cdot)-\rho_+^\ep(s,t,\cdot)\norm_{L^2(\Omega)}^2
\le 2\norm  \rho_+^\ep(s,\cdot,\cdot)\norm_{L^\infty(\Omega)} \norm \rho_+^\ep(s,t+h,\cdot)-\rho_+^\ep(s,t,\cdot)\norm_{L^1(\Omega)}\to0.
\end{align*}
Next we show that $\{\rho_+^\ep(s,\cdot,\cdot)_{t\in[s,s+T]}\}_{\ep>0}$ is a Cauchy sequence in   $C^0([s,s+T];L^2(\Omega))$. For this purpose, take any $\ep,\tilde{\ep}>0$ and apply \eqref{giro1} to $\rho_+^\ep(s,\cdot,\cdot)$ and $\rho_+^{\tilde{\ep}}(s,\cdot,\cdot)$.  
Then, the reasoning from \eqref{giro1} to \eqref{kae2} yields  
\begin{align*}
 &\int_{s-1}^{s+T}\int_\Omega\Big[\{\rho_+^{\tilde{\ep}}(s,t,x)-\rho_+^\ep(s,t,x)\}^2f'(t) + (\nabla\cdot v_+(t,x))\{\rho_+^{\tilde{\ep}}(s,t,x)-\rho_+^\ep(s,t,x)\}^2f(t)\\
& \qquad  -2\{R_+^{\tilde{\ep}}(s,t,x)-R_+^\ep(s,t,x)\}\{\rho_+^{\tilde{\ep}}(s,t,x)-\rho_+^\ep(s,t,x)\}f(t)\Big]dxdt=0. 
 \end{align*}
Hence, we see that $\norm \rho_+^{\tilde{\ep}}(s,t,\cdot)-\rho_+^\ep(s,t,\cdot)\norm_{L^2(\Omega)}^2$ (we know that it is continuous) is weakly $t$-differentiable, which leads to  
\begin{align*}
&\norm \rho_+^{\tilde{\ep}}(s,t,\cdot)-\rho_+^\ep(s,t,\cdot)\norm_{L^2(\Omega)}^2
=\norm \rho_0^{\tilde{\ep}}-\rho_0^\ep\norm_{L^2(\Omega)}^2
+\int_{s}^t\int_\Omega   \Big[
 (\nabla\cdot v_+({r},x))\{\rho_+^{\tilde{\ep}}(s,r,x)-\rho_+^\ep(s,r,x)\}^2\\
 &\quad -2\{R_+^{\tilde{\ep}}(s,r,x)-R_+^\ep(s,r,x)\}\{\rho_+^{\tilde{\ep}}(s,r,x)-\rho_+^\ep(s,r,x)\}
\Big]dxdr \quad \mbox{for every $t\in[s,s+T]$.}
\end{align*}
Let $\nu>0$ be arbitrary. 
Since $R_+^{\tilde{\ep}}(s,\cdot,\cdot),R_+^\ep(s,\cdot,\cdot)$ converges to $0$ as $\ep,\tilde{\ep}\to0$ in the $L^1$-norm, we find $\ep_\nu>0$ such that if $\ep,\tilde{\ep}\in(0,\ep_\nu)$ it holds that 
\begin{align}\label{kita}
\sup_{t\in[s,s+T]}\int_{s}^{t}\int_{\Omega}|2\{R_+^{\tilde{\ep}}(s,r,x)-R_+^\ep(s,r,x)\}\{\rho_+^{\tilde{\ep}}(s,r,x)-\rho_+^\ep(s,r,x)\}|dxdr<\nu,
\end{align}
as well as $\norm \rho_0^{\tilde{\ep}}-\rho_0^\ep\norm_{L^2(\Omega)}^2<\nu$.  
Therefore, Gronwall's inequality implies that for all $\ep,\tilde{\ep}\in(0,\ep_\nu)$,
\begin{align*}
\sup_{t\in[s,s+T]}\norm \rho_+^{\tilde{\ep}}(s,t,\cdot)-\rho_+^\ep(s,t,\cdot)\norm_{L^2(\Omega)}^2
\le 2\nu e^{\norm \nabla\cdot v\norm_{L^1([s,s+T];L^\infty(\Omega))}}.
\end{align*}
\indent Let $\tilde{\rho}(s,\cdot,\cdot)\in C^0([s,s+T];L^2(\Omega))$ be the limit of $\{\rho_+^\ep(s,\cdot,\cdot)|_{t\in[s,s+T]}\}_{\ep>0}$. 
Since $\rho_+^\ep(s,\cdot,\cdot)|_{t\in[s,s+T]}\to \rho(s,\cdot,\cdot)|_{t\in[s,s+T]}$ in $L^2([s,s+T]\times\Omega)$ as $\ep\to 0$, there exists a null set $N\subset[s,s+T]$ such that $\rho_+^\ep(s,t,\cdot)\to \rho(s,t,\cdot)$ in $L^2(\Omega)$ as $\ep\to0$ for all $t\in[s-1,s+T]\setminus N$.   
Changing the value of $\rho(s,\cdot,\cdot)$ to be $\tilde{\rho}(s,\cdot,\cdot)$ on the null set $N\times\Omega$,  we  identify $\rho(s,\cdot,\cdot)$ with $\tilde{\rho}(s,\cdot,\cdot)$ to be an element of $C^0([s,s+T];L^2(\Omega))$, confirming  
\begin{align}\label{kuri}
\sup_{t\in[s,s+T]}\norm\rho_+^\ep(s,t,\cdot)-\rho(s,t,\cdot)\norm_{L^2(\Omega)} \to0\quad\mbox{ as $\ep\to0$}.  
\end{align}
Since $T>0$ is arbitrary, we conclude the proof together with the same reasoning for $[s-T,s]$  with $\rho_-^\ep(s,\cdot,\cdot)$. 
\end{proof}
\noindent  Due to the $C^0_tL^2_x$-regularity, $\rho(s,t,\cdot)$ makes sense for every $t\in\R$, which guarantees that the generalized flow map may take any initial time from $\R$.  
We will see that the weak convergence of approximate solutions $\rho^k(s,\cdot,\cdot)$ of \eqref{1LT-2} to $\rho$ is indeed strong.   
\begin{Prop}\label{LT-L2converge}
Let $s\in\R$ be arbitrary. Let $\{\rho^{k}(s,\cdot,\cdot)\}_{k\in\N}$, $\rho^{k}(s,t,x)=\rho^k_0(X^k(s,t,x))$ be the sequence of the approximate solution defined in \eqref{LT-approx} with $\rho^k_0$ given as either (I1) or (I2). 
Let   $\rho(s,\cdot,\cdot)$ be a unique weak solution of \eqref{LT} that is identified with an element of $C^0(\R;L^2(\Omega))$.  
Then, it {holds} that 
\begin{align*}
\sup_{t\in[T_0,T_1]}\norm \rho^{k}(s,t,\cdot)- \rho(s,t,\cdot)\norm_{L^2(\Omega)}\to0\quad \mbox{ as $k\to\infty$, \,\,\,\, for each $T_0<T_1$.} 
\end{align*} 
\end{Prop}

\begin{proof} 
Recall the definition of $\rho_\pm$, $\rho_\pm^\ep$ and $\rho^\ep_0$ stated just before  Lemma \ref{LT-commu} with the representative $\rho$ found in Proposition \ref{LT-unique}, where $\rho$ is extended to be $0$ outside $\Omega$.  
Let $s\in\R$ be arbitrary. For each $T>0$ and each $\varphi\in C^\infty([s,s+T]\times K)$ with supp$(\varphi)\subset(s,s+T)\times K$, we have   
\begin{align*}
0=\int_{s}^{s+T}\int_K \Big\{\rho^k(s,t,x)\p_t\varphi(t,x) - v^k(t,x)\cdot \nabla \rho^k(s,t,x)\varphi(t,x) \Big\}dxdt, 
\end{align*}
and hence due to Lemma \ref{LT-commu} or \eqref{giro1} with $\rho^\ep(s,\cdot,\cdot):=\rho_+^\ep(s,\cdot,\cdot)$ for $t\in[s,s+T]$  (recall that $v_+=v$ within $[s,s+T]$),  
\begin{align*}
&\int_{s}^{s+T}\int_K (\rho^k(s,t,x)-\rho^\ep(s,t,x))\p_t\varphi(t,x)dxdt \\
&\qquad = \int_{s}^{s+T}\int_K \{v^k(t,x)\cdot \nabla \rho^k(s,t,x)-v(t,x)\cdot \nabla \rho^\ep(s,t,x) -R_+^\ep(s,t,x)\}\varphi(t,x) \,dxdt.
\end{align*}
Hence, $\rho^k(s,\cdot,\cdot)-\rho^\ep(s,\cdot,\cdot)$ is weakly $t$-differentiable with the weak derivative belonging to $ L^1([s,s+T]\times K)$: 
$$\p_t\{\rho^k(s,t,x)-\rho^\ep(s,t,x)\}=-\{v^k(t,x)\cdot \nabla \rho^k(s,t,x)-v(t,x)\cdot \nabla \rho^\ep(s,t,x) -R_+^\ep(s,t,x)\}.$$
Then, for a.e. $x\in K$, the function $\rho^k(s,\cdot,x)-\rho^\ep(s,\cdot,x)$ is differentiable for a.e. $t\in [s,s+T]$ in the classical sense; the (classical)  derivative $\frac{\p}{\p t}(\rho^k(s,t,x)-\rho^\ep(s,t,x))$ exists for a.e. $(t,x)\in[s,s+T]\times K$ and coincides with the weak $t$-derivative. 
Therefore, we have a.e. in $[s,s+T]\times K$,
\begin{align*}
&\p_t\{ (\rho^k(s,t,x)-\rho^\ep(s,t,x))^2\} +v^k(t,x)\cdot\nabla\{ (\rho^k(s,t,x)-\rho^\ep(s,t,x))^2\}\\
&\quad +2\{(v^k(t,x)-v(t,x))\cdot \nabla \rho^\ep(s,t,x)   -R_+^\ep(s,t,x) \}(\rho^k(s,t,x)-\rho^\ep(s,t,x))=0.
\end{align*}
For each $\varphi\in C^\infty([s,s+T]\times K)$ with supp$(\varphi)\subset(s,s+T)\times K$, we have   
\begin{align*}
&\int_{s}^{s+T}\int_K\Big[ (\rho^k(s,t,x)-\rho^\ep(s,t,x))^2 \p_t\varphi(t,x) - v^k(t,x)\cdot\nabla\{(\rho^k(s,t,x)-\rho^\ep(s,t,x))^2\} \varphi(t,x)\\\nonumber
&-2\{(v^k(t,x)-v(t,x))\cdot \nabla \rho^\ep(s,t,x)   -R_+^\ep(s,t,x) \}(\rho^k(s,t,x)-\rho^\ep(s,t,x))\varphi(t,x)\Big]dxdt=0. 
\end{align*}
{Let $\{\chi^m\}_{m\in\N}\subset C^\infty_0(K)$, $0\le\chi^m\le1$, be a smooth approximation of $\chi_K$ (the indicator function of $K$), converging a.e. pointwise. Taking $\varphi(t,x)=\chi^m(x)f(t)$ with $f\in C^\infty([s,s+T])$ with supp$(f)\subset(s,s+T)$ and sending $m\to\infty$ with $k$ fixed, we have} 
\begin{align*}
&\int_{s}^{s+T}\int_K\Big[ (\rho^k(s,t,x)-\rho^\ep(s,t,x))^2 f'(t) - v^k(t,x)\cdot\nabla\{(\rho^k(s,t,x)-\rho^\ep(s,t,x))^2\} f(t)\\\nonumber
&-2\{(v^k(t,x)-v(t,x))\cdot \nabla \rho^\ep(s,t,x)   -R_+^\ep(s,t,x) \}(\rho^k(s,t,x)-\rho^\ep(s,t,x))f(t)\Big]dxdt=0, 
\end{align*}
which leads to 
\begin{align*}
&\int_{s}^{s+T}\Big[\int_K (\rho^k(s,t,x)-\rho^\ep(s,t,x))^2dx\Big] f'(t)dt   = -\int_{s}^{s+T}\int_K\Big[ (\nabla\cdot v^k(t,x))\{(\rho^k(s,t,x)-\rho^\ep(s,t,x))^2\} \\
&\quad -2\{(v^k(t,x)-v(t,x))\cdot \nabla \rho^\ep(s,t,x)   -R_+^\ep(s,t,x) \}(\rho^k(s,t,x)-\rho^\ep(s,t,x))\Big]f(t)dxdt.  
\end{align*}
Hence, $\norm \rho^k(s,t,\cdot)-\rho^\ep(s,t,\cdot)\norm_{L^2(K)}^2$ is weakly $t$-differentiable, and absolutely continuous  on $[s,s+T]$, where we already know that  $\norm \rho^k(s,s,\cdot)-\rho^\ep(s,s,\cdot)\norm_{L^2(K)}^2=\norm \rho_0^k-\rho_0^\ep\norm_{L^2(K)}^2 $ due to Proposition \ref{LT-unique}:  
 \begin{align}\label{kips}
&\norm \rho^k(s,t,\cdot)-\rho^\ep(s,t,\cdot)\norm_{L^2(K)}^2
\\\nonumber
&\quad =\norm \rho_0^k-\rho_0^\ep\norm_{L^2(K)}^2 
+\int_{s}^t\int_K (\nabla\cdot v(r,x))(\rho^k(s,r,x)-\rho^\ep(s,r,x))^2\,dxdr \\\nonumber
&\quad\quad+\int_{s}^t\int_K (\nabla\cdot v^k(r,x)-\nabla\cdot v(r,x))(\rho^k(s,r,x)-\rho^\ep(s,r,x))^2\,dxdr\\\nonumber
&\quad \quad -2\int_{s}^t\int_K(v^k(r,x)-v(r,x))\cdot (\nabla \rho^\ep(s,r,x)) (\rho^k(s,r,x)-\rho^\ep(s,r,x))dxdr\\\nonumber
&\quad \quad +2\int_{s}^t\int_KR_+^\ep(s,r,x) (\rho^k(s,r,x)-\rho^\ep(s,r,x))dxdr\quad\mbox{ for all  $t\in[s, s+T]$}.
 \end{align}
{\indent Here \eqref{kuri} also holds with $K$ in place of $\Omega$: outside $\Omega$, $\rho^\ep(s,t,\cdot)$ is uniformly bounded and supported in $\{x\in K\setminus\Omega\,|\,\operatorname{dist}(x,\overline{\Omega})\le\ep\}$, whose measure tends to zero because $\meas(\partial\Omega)=0$.}
In view of  $\rho(s,\cdot,\cdot)\equiv0$ outside $\Omega$  and \eqref{kuri}, we have 
\begin{align}\label{sop1}
\sup_{t\in[s,s+T]}\norm \rho^\ep(s,t,\cdot)\norm_{L^2(K\setminus\Omega)}\le \sup_{t\in[s,s+T]}\norm \rho^\ep(s,t,\cdot)-\rho(s,t,\cdot)\norm_{L^2(K)}\to0\quad\mbox{as $\ep\to0$}. 
\end{align}
Similarly, we have\footnote{\eqref{sop2} is true in both the cases (I1), (I2) to define  $\rho^k_0$ (see the beginning of Section 3).} 
\begin{align}\label{sop2}
\sup_{t\in[s,s+T]}\Big|\norm \rho^k(s,t,\cdot)\norm_{L^2(K\setminus\Omega)}-\norm \rho_0^k\norm_{L^2(K\setminus\Omega)}\Big|\to0\quad\mbox{as $k\to\infty$}.
\end{align}
In fact, recall that $v$ is extended by zero outside $\Omega$, and supp$(v^k)$ shrinks to $\overline{\Omega}$  as $ k\to\infty$: with the mollification radius
$r_k\downarrow0$, supp$(v^k(t,\cdot))\subset
\{x\in\mathbb R^3\,|\,\operatorname{dist}(x,\overline{\Omega})\le r_k\}$ for all $t\in[s,s+T]$. 
Set $D_k:=
(K\setminus\Omega)\cap
\left\{
x\in\mathbb R^3\,|\,
\operatorname{dist}(x,\overline{\Omega})\le r_k
\right\}$.
If $x\in (K\setminus\Omega)\setminus D_k$, then we have $X^k(r,t,x)\equiv x$ and  $\rho^k(s,t,x)\equiv \rho_0^k(x)$. 
Consequently, it holds that supp$(\rho^k(s,t,\cdot)-\rho_0^k)\cap (K\setminus\Omega)\subset D_k$, and hence 
$$
\sup_{t\in[s,s+T]}
\norm \rho^k(s,t,\cdot)-\rho_0^k
\norm_{L^2(K\setminus\Omega)}
\le  2M\meas(D_k)^{\frac{1}{2}},
$$
where $M$ is a $k$-independent $L^\infty$-bound of $\rho^k,\rho^k_0$.
Since $\meas(\partial\Omega)=0$ by assumption, 
and $r_k\to0$ as $k\to\infty$, we have $\meas(D_k)\to 0$ as $k\to\infty$.  
This shows  \eqref{sop2}. 
Take any $\nu>0$.  
We find $\ep_\nu>0$ for which:  due to  \eqref{kuri},  if $0<\ep<\ep_\nu$, it holds that  
\begin{align}\label{well}
&\sup_{t\in[s,s+T]}\norm \rho^\ep(s,t,\cdot)-\rho(s,t,\cdot)\norm_{L^2(K)}<\nu;
\end{align} 
due to \eqref{sop1},  if $0<\ep<\ep_\nu$ and $t\in[s,s+T]$, it holds that  for all $k\in\N$,
\begin{align}\label{sop3}
&\norm \rho^k(s,t,\cdot)-\rho^\ep(s,t,\cdot)\norm_{L^2(K)}^2
=\norm \rho^k(s,t,\cdot)-\rho^\ep(s,t,\cdot)\norm_{L^2(\Omega)}^2+\int_{K\setminus\Omega} \rho^k_0(x)^2dx\\\nonumber
&\qquad +\int_{K\setminus\Omega}\{\rho^k(s,t,x)^2- \rho^k_0(x)^2\}dx
+\int_{K\setminus\Omega}\{\rho^\ep(s,t,x)-2\rho^k(s,t,x)\}\rho^\ep(s,t,x)dx\\\nonumber
&\quad \ge \norm \rho^k(s,t,\cdot)-\rho^\ep(s,t,\cdot)\norm_{L^2(\Omega)}^2+\int_{K\setminus\Omega} \rho^k_0(x)^2dx+\int_{K\setminus\Omega}\{\rho^k(s,t,x)^2- \rho^k_0(x)^2\}dx-\nu;
\end{align}
since $\norm \rho_0^\ep-\rho_0\norm_{L^2(\Omega)}\to0$ and $\rho_0^\ep(x)\to0$ for all $x\in K\setminus\overline{\Omega}$ as $\ep\to0$,  if $0<\ep<\ep_\nu$, it holds that   for all $k\in\N$,
\begin{align}\label{sop4}
\norm \rho^k_0&-\rho^\ep_0\norm_{L^2(K)}^2
=\int_{K\setminus\Omega} \rho^k_0(x)^2dx
+\int_{K\setminus\Omega} \{\rho^\ep_0(x)-2\rho^k_0(x)\}\rho_0^\ep(x)dx\\\nonumber
&+\int_{\Omega} \{\rho_0^k(x)-\rho^\ep_0(x)\}\{\rho_0^k(x)-\rho_0(x)\}dx+\int_{\Omega} \{\rho_0^\ep(x)-\rho^k_0(x)\}\{\rho_0^\ep(x)-\rho_0(x)\}dx\\\nonumber
\le& \int_{K\setminus\Omega} \rho^k_0(x)^2dx
+\int_{\Omega} \{\rho_0^k(x)-\rho^\ep_0(x)\}\{\rho_0^k(x)-\rho_0(x)\}dx+\nu;
\end{align}
since $R_+^\ep(s,\cdot,\cdot)\to0$ in $L^1$ as $\ep\to0$,  if $0<\ep<\ep_\nu$, it holds that   for all $k\in\N$,
\begin{align}\label{tete3}
\sup_{t\in[s,s+T]}\Big|2\int_{s}^t\int_KR_+^\ep(s,r,x) (\rho^k(s,r,x)-\rho^\ep(s,r,x))dxdr\Big|<\nu. 
\end{align}
Fix  any $\ep\in(0,\ep_\nu)$. Then, in view of \eqref{sop2} and $\norm \rho_0^k-\rho_0\norm_{L^2(\Omega)}\to0$ as $k\to\infty$, we find $k_\nu\in\N$ such that:  if $k\ge k_\nu$  and $t\in[s,s+T]$ the estimate \eqref{sop3} and \eqref{sop4} lead to   
\begin{align*}
\norm \rho^k(s,t,\cdot)-\rho^\ep(s,t,\cdot)\norm_{L^2(K)}^2 &\ge \norm \rho^k(s,t,\cdot)-\rho^\ep(s,t,\cdot)\norm_{L^2(\Omega)}^2+\int_{K\setminus\Omega} \rho^k_0(x)^2dx-2\nu,\\
\norm \rho^k_0-\rho^\ep_0\norm_{L^2(K)}^2&\le \int_{K\setminus\Omega} \rho^k_0(x)^2dx+2\nu;
\end{align*}
 if $k\ge k_\nu$, we have   
\begin{align*}
\sup_{t\in[s,s+T]}\Big|\int_{s}^t\int_K (\nabla\cdot v^k(r,x)&-\nabla\cdot v(r,x))(\rho^k(s,r,x)-\rho^\ep(s,r,x))^2\,dxdr\\\nonumber
 -2\int_{s}^t\int_K(v^k(r,x)&-v(r,x))\cdot (\nabla \rho^\ep(s,r,x)) (\rho^k(s,r,x)-\rho^\ep(s,r,x))dxdr  \Big|<\nu,
\end{align*}
where $v^k\to v$ and $\nabla\cdot v^k\to\nabla\cdot v$ as $k\to\infty$.  
Therefore, \eqref{kips}  and Gronwall's inequality yield for all $k\ge k_\nu$  and all $t\in[s,s+T]$, 
\begin{align*}
\norm \rho^k(s,t,\cdot)-\rho^\ep(s,t,\cdot)\norm_{L^2(\Omega)}^2
\le &6\nu +\int_{s}^t\norm \nabla\cdot v(r,\cdot) \norm_{L^\infty(\Omega)}\norm\rho^k(s,r,\cdot)-\rho^\ep(s,r,\cdot)\norm_{L^2(\Omega)}^2\,dr,\\
\norm \rho^k(s,t,\cdot)-\rho^\ep(s,t,\cdot)\norm_{L^2(\Omega)}^2\le& 6\nu e^{  \norm \nabla\cdot v\norm_{L^1([s,s+T];L^\infty(\Omega))}}.
\end{align*}
By \eqref{well}, we have for all $k\ge k_\nu$,
\begin{align*}
\sup_{t\in[s,s+T]}\norm \rho^k(s,t,\cdot)-\rho(s,t,\cdot)\norm_{L^2(\Omega)}<\nu+\Big\{   6\nu e^{  \norm \nabla\cdot v\norm_{L^1([s,s+T];L^\infty(\Omega))}} \Big\}^\2.
\end{align*}
A similar argument shows that for any $T>0$,  
\begin{align*}
\sup_{t\in[s-T,s]}\norm \rho^k(s,t,\cdot)-\rho(s,t,\cdot)\norm_{L^2(\Omega)}<\nu+\Big\{   6\nu e^{  \norm \nabla\cdot v\norm_{L^1([s-T,s];L^\infty(\Omega))}} \Big\}^\2.
\end{align*}
Since $\nu>0$ is arbitrary, we conclude the proof.
\end{proof}
Next, we consider \eqref{LT} fixing $\rho_0\in L^\infty(\Omega)$ and varying  the initial time $s\in\R$, and analyze how $\rho(s,t,x)$ depends on $s$, where  $s$ is  the initial time  for the PDE and the evolution time for the ODE. 
We already know that $\rho:\R\times\R\times\Omega\to\R$ is well-defined as $\rho(s,\cdot,\cdot)\in L^\infty(\R\times\Omega)$ for every $s\in\R$ and $\rho(s,t,\cdot)$ is determined for every $t\in\R$.  
\begin{Prop}\label{LT-kasoku-2}
For each $s\in\R$, let $\{\rho^{k}(s,\cdot,\cdot)\}_{k\in\N}$ be the sequence of approximate solution of  \eqref{1LT-2} given as \eqref{LT-approx} and let  $\rho(s,\cdot,\cdot)$ be a unique weak solution of \eqref{LT} that is identified with an element of $C^0(\R;L^2(\Omega))$.  
Then, for each $t\in\R$, there exists a measurable function $u(\cdot,t,\cdot):\R\times\Omega\to\R$ such that $u(\cdot,t,\cdot)\in L^\infty(\R\times\Omega)$,  $u(s,t,\cdot)=\rho(s,t,\cdot)$ for every $s\in\R$,  and  for each $T_0<T_1$ it holds that   $\rho^k(\cdot,t,\cdot)|_{s\in[T_0,T_1]}\to u(\cdot,t,\cdot)|_{s\in[T_0,T_1]}$ in $L^2([T_0,T_1]\times\Omega)$  as $k\to\infty$. 
In particular, it holds that $\rho^k(s,t,\cdot) \to u(s,t,\cdot)$ in $L^2(\Omega)$ as $k\to\infty$ for every $s,t\in\R$.  
%
\end{Prop}

\begin{proof} $\rho^k(s,t,x)$ is $C^1$-smooth in all variables. 
For each $T>0$, we have $\norm \rho^{k}(\cdot,t,\cdot)\norm_{L^2([-T,T]\times\Omega)}^2\le 2T\meas(\Omega)M^2$ with $M:=\sup_k\|\rho_0^k\|_{L^\infty(K)}<\infty$ for all $k\in\N$. 
Hence, there exists a subsequence $\{\rho^{a_k}(\cdot,t,\cdot)\}_{k\in\N}\subset\{\rho^k(\cdot,t,\cdot)\}_{k\in\N}$ that converges weakly to some element of $ L^2([-T,T]\times\Omega)$ as $k\to\infty$. 
By a diagonal argument with respect to $T\in\N$, we find a subsequence $\{\rho^{b_k}(\cdot,t,\cdot)\}_{k\in\N}\subset\{\rho^k(\cdot,t,\cdot)\}_{k\in\N}$ and a measurable function $u(\cdot,t,\cdot):\R\times\Omega\to\R$ such that $\rho^{b_k}(\cdot,t,\cdot)|_{s\in[-T,T]} \wto u(\cdot,t,\cdot)|_{s\in[-T,T]}$ in $ L^2([-T,T]\times\Omega)$ as $k\to\infty$ for all $T>0$.  

Due to Proposition \ref{LT-L2converge}, we see that $(\rho^{b_k}(s,t,\cdot),\varphi)_{L^2(\Omega)}$ converges to $(\rho(s,t,\cdot),\varphi)_{L^2(\Omega)}$ as $k\to\infty$ for each $s\in\R$  and $\varphi\in L^2(\Omega)$. Hence,  $(\rho(s,t,\cdot),\varphi)_{L^2(\Omega)}$ is a bounded  measurable function of $s$ and satisfies for all $\psi\in L^2([-T,T])$, 
\begin{align*}
& \lim_{k\to\infty} \int_{-T}^T\int_\Omega \rho^{b_k}(s,t,x)\varphi(x)\psi(s)dxds=\int_{-T}^T\int_\Omega u(s,t,x)\varphi(x)\psi(s)dxds\\
 &\quad= \int_{-T}^T (u(s,t,\cdot),\varphi)_{L^2(\Omega)}\psi(s) ds 
 =  \int_{-T}^T(\rho(s,t,\cdot),\varphi)_{L^2(\Omega)}\psi(s)ds.
 \end{align*}
Since $\psi$ is arbitrary, there exists a null set $N_T(\varphi)\subset[-T,T]$ such that 
$$(u(s,t,\cdot),\varphi)_{L^2(\Omega)}=(\rho(s,t,\cdot),\varphi)_{L^2(\Omega)} \quad\mbox{for all $s\in [-T,T]\setminus N_T(\varphi)$}.$$
Since $L^2(\Omega)$ is separable, we have a dense set $\{\varphi^j\}_{j\in\N}\subset L^2(\Omega)$.  
Define the null set $N_T:=\cup_{j\in\N}N_T(\varphi_j)$. For each $\varphi\in L^2(\Omega)$, we have  
\begin{align*}
&(u(s,t,\cdot), \varphi)_{L^2(\Omega)}-(\rho(s,t,\cdot),\varphi)_{L^2(\Omega)} \\
&\qquad =(u(s,t,\cdot),\varphi-\varphi^j)_{L^2(\Omega)}-(\rho(s,t,\cdot),\varphi-\varphi^j)_{L^2(\Omega)}
\quad\mbox{for all $s\in [-T,T]\setminus N_T$},
\end{align*}
where a suitable choice of $j$ makes the right hand side arbitrarily close to $0$.  Hence, we find that 
$$u(s,t,\cdot)=\rho(s,t,\cdot)\quad\mbox{for all $\dis s\in \R\setminus N$ with $\dis N:=\bigcup_{T\in\N}N_T$}.$$
Thus $u(s,t,\cdot)=\rho(s,t,\cdot)$ for all  $s\in\R\setminus N$.

The weak convergence of $\{\rho^{b_k}(\cdot,t,\cdot)\}_{k\in\N}$ to $u(\cdot,t,\cdot)$ on $[-T,T]\times\Omega$ is in fact strong convergence, because 
\begin{align*}
&\Big|\norm \rho^{b_k}(\cdot,t,\cdot)\norm_{L^2([-T,T]\times\Omega)}^2-\norm u(\cdot,t,\cdot)\norm_{L^2([-T,T]\times\Omega)}^2\Big|\\
&\le \int_{-T}^T\Big[\int_\Omega |\rho^{b_k}(s,t,x)+u(s,t,x)||\rho^{b_k}(s,t,x)-u(s,t,x)| dx\Big]ds\to0\mbox{\quad as $k\to\infty$,}
\end{align*}
where $\norm\rho^{b_k}(s,t,\cdot)-u(s,t,\cdot)\norm_{L^2(\Omega)}^2=\norm\rho^{b_k}(s,t,\cdot)-\rho(s,t,\cdot)\norm_{L^2(\Omega)}^2$ a.e. $s\in\R$ is measurable and Proposition \ref{LT-L2converge} is used.
Suppose that we do not have $\rho^k(\cdot,t,\cdot)|_{s\in[-T,T]}\to u(\cdot,t,\cdot)|_{s\in[-T,T]}$ in $L^2([-T,T]\times\Omega)$  as $k\to\infty$ for some $T$. 
Then, we find $\theta>0$ and a subsequence $\{\rho^{a_k}\}_{k\in\N}\subset\{\rho^k\}_{k\in\N}$ such that  $ \norm  \rho^{a_k}(\cdot,t,\cdot)
-u(\cdot,t,\cdot)\norm_{L^2([-T,T]\times\Omega)}\ge\theta$ for all $k\in\N$. 
However, an argument similar to the above yields a subsequence  $\{\rho^{b_k}\}_{k\in\N}\subset\{\rho^{a_k}\}_{k\in\N}$ such that  $ \norm  \rho^{b_k}(\cdot,t,\cdot)
-u(\cdot,t,\cdot)\norm_{L^2([-T,T]\times\Omega)}\to 0$ as $k\to\infty$. This is a contradiction. 

Changing the value of $u(\cdot,t,\cdot)$ on the null set $N\times\Omega$ into $\rho(\cdot,t,\cdot)$, we confirm that $u(s,t,\cdot)=\rho(s,t,\cdot)$ for every $s\in\R$ with $ \norm  \rho^{k}(\cdot,t,\cdot)
-u(\cdot,t,\cdot)\norm_{L^2([-T,T]\times\Omega)}\to 0$ as $k\to\infty$ and 
$\norm \rho^k(s,t,\cdot)- u(s,t,\cdot)\norm_{L^2(\Omega)}=\norm \rho^k(s,t,\cdot)- \rho(s,t,\cdot)\norm_{L^2(\Omega)}\to 0$ as $k\to\infty$ for every $s\in\R$.   
\end{proof}
\noindent The representative $u(\cdot,t,\cdot)$ constructed in Proposition \ref{LT-kasoku-2} will serve as the generalized flow map for the initial time $t$ in the next section when the initial data is chosen as $\rho_0(x)=x_i$ ($i=1,2,3$). {For each fixed $t$, this representative is measurable in $(s,x)$ and agrees in $L^2(\Omega)$ with the solution at every $s$, rather than merely for a.e. $s$. In particular, we may work with each fixed time pair $(s,t)$ in the $L^2$ sense.}   
\setcounter{section}{3}
\setcounter{equation}{0}
\section{Generalized flow map and Reynolds transport theorem} 

In this section, we construct a generalized flow map and prove Theorems \ref{Thm-DL}, \ref{FL-preservative2} and \ref{main}, together with auxiliary lemmas and propositions required  for the limiting arguments.  
{In particular, we construct the common trimming in Definition 1 and show that the resulting image volumes, up to null sets, are determined by smooth approximation.}   

We now specialize the Section 3 to the case where the smooth initial data $\rho_0$ for \eqref{LT} is given as  
$$ \rho_0(x):=x_i\quad (i=1,2,3),$$
and deal with \eqref{LT-ODE}, \eqref{LT-approx}, \eqref{1LT-2} with $\rho_0^k(x)=x_i$ ($i=1,2,3$) for all $k\in\N$ (i.e., (I1)). 
Denote the corresponding approximate solutions $\rho^k$ of \eqref{1LT-2} by $X_i^k(s,t,x)$, and let $X_i(\cdot,t,\cdot)$ be the representative $u(\cdot,t,\cdot)$ obtained in Proposition \ref{LT-kasoku-2}. 
Writing  $X=(X_1,X_2,X_3)$, we already know that 
\begin{align*}
&\mbox{$ X^k(\cdot,t,\cdot)\to X(\cdot,t,\cdot)$ in $L^2([T_0,T_1]\times\Omega)^3$ for each $T_0<T_1$},\\
&\mbox{$ X^k(s,t,\cdot)\to X(s,t,\cdot)$ in $L^2(\Omega)^3$ for every $s,t\in\R$.}
\end{align*}  
Our first objective is to pass to the limit in the classical flow identity
\begin{align}\label{flow^k}
X^k(s,t,x)=x+\int_t^s v^k(r,X^k(r,t,x))\,dr, 
\end{align}
which leads to  the corresponding integral formula satisfied by $X$.  
{We take $X$ as constructed above, with values in $\overline{\Omega}$, and extend it by $X(s,t,x):=x$ for $x\notin\Omega$.}
For this purpose, we introduce the space-time maps $G^k(s,t,x):=(s,X^k(s,t,x))$, $G(s,t,x):=(s,X(s,t,x))$ and investigate measurability of the (inverse) image  under  $G$ and $X$. 
The following lemma is required not only for justification of regular co-moving volumes but also for the proof of convergence of $v^k(\cdot,X^k(\cdot,t,\cdot))$ to $v(\cdot,X(\cdot,t,\cdot))$.   
\begin{Lemma}\label{FL-image}
Let $t\in\R$ be arbitrary.  Let $A\subset\R\times\Omega$ and $B\subset\Omega$ be arbitrary  bounded measurable sets. Then, the following statements hold true:   
\begin{enumerate}
\item  There exists a null set $N^t_A\subset A$ for which the image  $\{ G(s,t,x) \,|\,  (s,x)\in A\setminus N_A^{t}\}$  is a bounded {Borel} measurable subset of $\R\times\overline{\Omega}$. 

\item  Let $N^t_A\subset A$ be a null set such that $\{ G(s,t,x) \,|\,  (s,x)\in A\setminus N^t_A\}$ is  a bounded measurable subset of $\R\times\overline{\Omega}$.  
Then, there exists a subsequence $\{G^{a_k}\}\subset\{G^k\}$  for which we find   for each $\delta>0$ and $\ep>0$  a closed set $A_\delta\subset \R^4$, an open set $O_\ep\subset\R\times K$ and $k_\ep\in\N$ such that  
\begin{itemize}
\item $A_\delta\subset A\setminus N^t_A\quad \mbox{with  $\meas(A_\delta)>\meas(A)-\delta$}$,
\item $ \meas(O_\ep)< \meas(\{G(s,t,x)  \,|\,(s,x)\in A_\delta\} )+\ep,$
\item $ \{G(s,t,x) \,|\,(s,x)\in A_\delta\} \mbox{ is a closed subset of $O_\ep$}$,
\item $\{G^{a_k}(s,t,x) \,|\,(s,x)\in A_\delta\}\mbox{ is a closed subset of $O_\ep$ for all $k\ge k_\ep$}$, 
\item $\meas(\{G^{a_k}(s,t,x)\,|\,  (s,x)\in  A_\delta\} )-\ep<  \meas(\{ G(s,t,x) \,|\,  (s,x)\in A\setminus N^t_A\})\quad\mbox{for all $k\ge k_\ep$}$.
\end{itemize}

\item  {Let $A\subset\R\times K$ be an arbitrary bounded measurable set.} There exists a bounded null set $N_A\subset {\R\times K}$ for which $A\cup N_A$ is Borel measurable, where the inverse image $\tilde{A}:=\{ (s,x)\in \R\times\Omega\,|\, G(s,t,x)\in A\cup N_A  \}$ is measurable. 
There exists  a subsequence $\{G^{b_k}\}\subset\{G^k\}$  for which we find  for each $\delta>0$ and $\ep>0$ a closed set $\tilde{A}_\delta\subset \tilde{A}$ and $k_\ep\in\N$ such that $\meas(\tilde{A}_\delta)>\meas(\tilde{A})-\delta$ and 
$$\meas(\{G^{b_k}(s,t,x)\,|\,  (s,x)\in  \tilde{A}_\delta\} )-\ep<  \meas(A)\quad\mbox{for all $k\ge k_\ep$}.$$ 

\item   Let $s\in\R$ be arbitrary.  There exists a null set $N_B^{s,t}\subset B$ for which the image  $X(s,t,B\setminus N_B^{s,t})$  is a {Borel} measurable subset of $\overline{\Omega}$. 


\item  Let $s\in\R$ be arbitrary. Let $N_B^{s,t}\subset B$ be a null set such that  $X(s,t,B\setminus N_B^{s,t})$  is a measurable subset of $\overline{\Omega}$. 
Then,  there exists  a subsequence $\{X^{c_k}\}\subset\{X^k\}$ for which we find  for each $\delta>0$ and $\ep>0$ a closed set $B_\delta\subset \R^3$, an open set $O_\ep\subset K$ and $k_\ep\in\N$ such that  
\begin{itemize}
\item $B_\delta\subset B\setminus N^{s,t}_B\quad \mbox{with  $\meas(B_\delta)>\meas(B)-\delta$}$,
\item $\meas(O_\ep)< \meas(X(s,t,B_\delta) )+\ep$, 
\item $X(s,t,B_\delta) \mbox{ is a closed subset of $O_\ep$} $,
\item $X^{c_k}(s,t,B_\delta)\mbox{ is a closed subset of $O_\ep$ for all $k\ge k_\ep$},$
\item $\meas(X^{c_k}(s,t, B_\delta ))-\ep<  \meas( X(s,t,B\setminus N^{s,t}_B))\quad\mbox{for all $k\ge k_\ep$}.$
\end{itemize}

\item   Let $s\in\R$ be arbitrary. {Let $B\subset K$ be an arbitrary bounded measurable set.}
There exists a null set $N_B\subset {K}$ for which $B\cup N_B$ is Borel measurable, where the inverse image $\tilde{B}:=X(s,t,\cdot)^{-1}(B\cup N_B )$ is measurable.     There exists  a subsequence $\{X^{d_k}\}\subset\{X^k\}$ for which we find  for each $\delta>0$ and $\ep>0$ a closed set $\tilde{B}_\delta\subset \tilde{B}$ and $k_\ep\in\N$ such that $\meas(\tilde{B}_\delta)>\meas(\tilde{B})-\delta$ and 
$$\meas({X^{d_k}(s,t,\tilde{B}_\delta)})-\ep<  \meas(B)\quad\mbox{for all $k\ge k_\ep$}.$$ 

\end{enumerate}
{In assertions 1 and 4, the trimmed images are countable unions of compact sets, and hence Borel measurable.}
\end{Lemma}
\begin{proof}
1. Take $s_0<s_1$ such that $A\subset[s_0,s_1]\times\Omega$. 
Since $X_i^k(\cdot,t,\cdot)\to X_i(\cdot,t,\cdot)$ in $L^2([s_0,s_1]\times\Omega)$ as $k\to\infty$, we find a subsequence $\{G^{a_k}\}\subset\{G^k\}$ and a null set $N_1{\subset}[s_0,s_1]\times\Omega$ such that   $G^{a_k}(s,t,x)\to G(s,t,x)$ as $k\to\infty$ for each $(s,x)\in A\setminus N_1$. 
If $x\in \Omega$ and $X^{a_k}(s,t,x)\to y$ as $k\to\infty$, we necessarily have $y\in \overline{\Omega}$. Indeed, if not,  the point $X^{a_k}(s,t,x)$ stays in an open ball $B_\nu(y)\subset \R^3\setminus\overline{\Omega}$ for all sufficiently large $k$; however, $v^{a_k}|_{B_\nu(y)}\equiv0$ for all sufficiently large $k$ and the solution $X^{a_k}(r,t,x)$ of $\gamma'(r)=v^{a_k}(r,\gamma(r))$ with $\gamma(t)=x$ cannot enter $B_\nu(y)$, which is a contradiction.   
{Hence, $\{G(s,t,x)\mid(s,x)\in A\setminus N_1\}\subset[s_0,s_1]\times\overline{\Omega}$.}  

Due to Egorov's theorem, for each $l\in\N$, there exists a measurable set $J^\circ_l\subset A$ such that 
$$\meas(A\setminus J^\circ_l)< \frac{1}{2l},\quad \mbox{$G^{a_k}(\cdot,t,\cdot)\to G(\cdot,t,\cdot)$ uniformly on $J^\circ_l$ as $k\to\infty$.}$$ 
Due to the regularity of the Lebesgue measure, there exists a closed set $J_l\subset J^\circ_l$ such that $\meas(J^\circ_l\setminus J_l)<\frac{1}{2l}$ and 
$$\meas(A\setminus J_l)< \frac{1}{l},\quad \mbox{$G^{a_k}(\cdot,t,\cdot)\to G(\cdot,t,\cdot)$ uniformly on $J_l$ as $k\to\infty$.}$$ 
Define  
$$J^\ast:=\bigcup_{l\in\N} J_l,\quad N_2:=A\setminus  J^\ast=\bigcap_{l\in\N}(A\setminus J_l).$$
Then, we see that $N_2$ is a null subset of $A$ and  
\begin{align*}
 \{ G(s,t,x) \,|\,  (s,x)\in A\setminus N_2=J^\ast\}= \bigcup_{l\in\N}  \{ G(s,t,x) \,|\,  (s,x)\in J_l\}. 
 \end{align*}
 Since $G^{a_k}(\cdot,t,\cdot)$ is continuous, $G(\cdot,t,\cdot)|_{J_l}$ being the limit of uniform convergence of $\{G^{a_k}(\cdot,t,\cdot)|_{J_l}\}_{k\in\N}$ is also continuous. 
 Hence,  $\{ G(s,t,x) \,|\,  (s,x)\in J_l\}$ is {compact and Borel measurable} for each $l\in\N$. 
 Thus, we conclude the assertion 1 with $N^t_A:=N_2$. 

2. Let $\{G^{a_k}\}\subset\{G^k\}$ be  a subsequence such that   $G^{a_k}(s,t,x)\to G(s,t,x)$ as $k\to\infty$ a.e. pointwise on $ A$.
Egorov's theorem implies that for each $\delta>0$ there exists a measurable set $A_\delta^\circ\subset A\setminus N_A^t$ such that $\meas(A_\delta^\circ)> \meas(A)-\frac{\delta}{2}$ and $G^{a_k}(\cdot,t,\cdot)|_{A_\delta^\circ}\to G(\cdot,t,\cdot)|_{A_\delta^\circ}$ uniformly as $k\to\infty$.  
 The regularity of the Lebesgue measure implies that there exists a closed set $A_\delta\subset A_\delta^\circ$ such that  $\meas(A_\delta)> \meas(A_\delta^\circ)-\frac{\delta}{2}> \meas(A)-\delta$. 
 Since $G(\cdot,t,\cdot)|_{A_\delta^\circ}$ is continuous, 
  the set $\{G(s,t,x)  \,|\,(s,x)\in A_\delta\} $ is closed and hence {Borel} measurable. 
  The regularity of the Lebesgue measure implies that for each $\ep>0$ there exists an open set $O_\ep$ such that 
  \begin{align*}
 \{G(s,t,x)  \,|\,(s,x)\in A_\delta\}\subset O_\ep,\quad   \meas(O_\ep)< \meas(\{G(s,t,x)  \,|\,(s,x)\in A_\delta\} )+\ep.
 \end{align*}
 Since  $\{G(s,t,x)  \,|\,(s,x)\in A_\delta\} $ is compact, there exists $\nu_\ep>0$ such that $B_{\nu_\ep}(z)\subset O_\ep$ for all $z\in \{G(s,t,x)  \,|\,(s,x)\in A_\delta\}$. 
Due to the uniform convergence, there exists $k_\ep\in\N$ such that for all $k\ge k_\ep$ we have $\sup_{(s,x)\in A_\delta}|G^{a_k}(s,t,x)-G(s,t,x)|<\nu_\ep/2$.
This implies that $ \{G^{a_k}(s,t,x)  \,|\,(s,x)\in A_\delta\}\subset O_\ep$ for all $k\ge k_\ep$.  
Thus, we have for all $k\ge k_\ep$,
$$\meas ( \{G^{a_k}(s,t,x)  \,|\,(s,x)\in A_\delta\}) <\meas(O_\ep)< \meas(\{G(s,t,x)  \,|\,(s,x)\in A_\delta\} )+\ep.$$

3. The regularity of the Lebesgue measure implies that there exists a bounded null set $N_A\subset {\R\times K}$ such that $A\cup N_A$ is Borel measurable.  
The inverse image $\tilde{A}:=\{(s,x)\in\R\times\Omega\,|\,G(s,t,x)\in A\cup N_A\}$ is measurable as $G(\cdot,t,\cdot):\R\times \Omega\to \R^4$ is {measurable}. 
Due to the assertion 1, there exists a null set $N$ such that $\{G(s,t,x)\,|\, (s,x)\in \tilde{A}\setminus N\}$ is  a {Borel} measurable subset of $A\cup N_A$. 
Due to the assertion 2, there exist a subsequence  $\{G^{b_k}\}\subset\{G^k\}$ for which we find for each $\delta>0$ and $\ep>0$ a closed set $\tilde{A}_\delta\subset \tilde{A}\setminus N$ and $k_\ep\in\N$ such that $\meas(\tilde{A}_\delta)>\meas(\tilde{A})-\delta$ and for all $k\ge k_\ep$,
$$\meas(\{G^{b_k}(s,t,x)\,|\,  (s,x)\in  \tilde{A}_\delta\} )-\ep <  \meas(\{ G(s,t,x) \,|\,  (s,x)\in \tilde{A}\setminus N\})\le \meas(A\cup N_A)=\meas(A).$$ 

4--6.  We may follow the same reasoning as above, where we note that $\norm X_i^k(s,t,\cdot) - X_i(s,t,\cdot)\norm_{L^2(\Omega)}\to0$ as $k\to\infty$ for each $s\in\R$.
  
\end{proof}
\begin{Lemma}\label{FL-image2}
Let $s, t\in\R$ be arbitrary.  Let $A\subset\R\times{K}$ and $B\subset{K}$ be arbitrary  bounded null sets. Then, the following statements hold true:   
\begin{enumerate}
\item  There exists a bounded null set $N_A\subset \R\times K$ for which $A\cup N_A$ is Borel measurable and the inverse image $\{ (s,x)\in \R\times\Omega\,|\, G(s,t,x)\in A\cup N_A  \}$ is a null set. 

\item  There exists a null set $N_B\subset {K}$ for which $B\cup N_B$ is Borel measurable and the inverse image $X(s,t,\cdot)^{-1}(B\cup N_B )$ is a null set. 

%
\end{enumerate}
\end{Lemma}
\begin{proof}
1. Set $\tilde{A}:=\{ (s,x)\in \R\times\Omega\,|\, G(s,t,x)\in A\cup N_A  \}$. By the assertion 3 of Lemma \ref{FL-image},  there exists a subsequence $\{G^{a_k}\}\subset\{G^k\}$  for which we find   for each $\delta>0$ and $\ep>0$ a closed set $\tilde{A}_\delta\subset \tilde{A}$ and $k_\ep\in\N$ such that $\meas(\tilde{A}_\delta)>\meas(\tilde{A})-\delta$ and 
\begin{align}\label{bruno}
\meas(\{G^{a_k}(s,t,x)\,|\,  (s,x)\in  \tilde{A}_\delta\} )-\ep<  \meas(A)=0\quad\mbox{for all $k\ge k_\ep$}.
\end{align}
Let $\tilde{A}_\delta(s)$ denote the cross section of $\tilde{A}_\delta$ at $s\in\R$, where $\tilde{A}_\delta(s)\subset\Omega$ is closed and hence measurable (possibly empty). 
For a sequence $\{\ep_j\}_{j\in\N}$ tending to zero as ${j}\to\infty$, we apply  \eqref{bruno} with $\ep=\ep_j$. Then, we find a sequence $\{k_j\}_{j\in\N}\subset\N$ such that 
$$\meas(\{G^{a_{k_j}}(s,t,x)\,|\,  (s,x)\in  \tilde{A}_\delta\} )=\int_\R \meas(X^{a_{k_j}}(s,t,\tilde{A}_\delta(s)))\,ds    <\ep_j\to0  \quad\mbox{as $j\to\infty$}.$$ 
Hence, up to a subsequence, $ \meas(X^{a_{k_j}}(s,t,\tilde{A}_\delta(s)))\to 0$ as $j\to\infty$ for a.e. $s\in\R$. 
Set $E^j_s:=X^{a_{k_j}}(s,t,\tilde{A}_\delta(s))$.  Then, we have 
$\tilde{A}_\delta(s)=X^{a_{k_j}}(t,s,E^j_s)$ and by Proposition \ref{classical3} and  Gronwall's inequality\footnote{It holds that $\norm \nabla\cdot v^k\norm_{L^1([t_1,t_2];L^\infty(K))}\le\norm \nabla\cdot v\norm_{L^1([t_1-\ep_k,t_2+\ep_k];L^\infty(K))}$ with $\ep_k$ the mollification parameter for $v^k$.},
\begin{align*}
\meas(X^{a_{k_j}}(\tilde{t},s,E^j_s))&=\meas(E^j_s)+\int_s^{\tilde{t}}\int_{X^{a_{k_j}}(r,s,{E^j_s})}\nabla\cdot v^{a_{k_j}}(r,x)dxdr\mbox{\quad for all $\tilde{t}\in\R$},\\
\meas(\tilde{A}_\delta(s))&=\meas(X^{a_{k_j}}(t,s,E^j_s))\le \meas(E^j_s)e^{\norm \nabla\cdot v^{a_{k_j}}\norm_{L^1(I^{s,t};L^\infty({K}))}}\to 0\quad\mbox{as $j\to\infty$}. 
\end{align*}
Therefore, we see that $\meas(\tilde{A}_\delta(s))=0$ for a.e. $s\in\R$ and $\tilde{A}_\delta$ is a null set. 
Since $\meas(\tilde{A})-\delta<\meas(\tilde{A}_\delta)=0$ and $\delta>0$ is arbitrary, we conclude that $\meas(\tilde{A})=0$.  

2. We may follow the same reasoning with the assertion 6 of Lemma \ref{FL-image}. 
\end{proof}

\begin{Lemma}\label{FL-product}
1. Let $t\in\R$, $S>0$ and $w\in L^1([t-S,t+S]\times K)$ be arbitrary. 
Let $\{w^k\}_{k\in\N}\subset  C^1([t-S,t+S]\times K)$ be such that $w^k\to w$ in $L^1([t-S,t+S]\times K)$ as $k\to\infty$.    
Then, the function $w(\cdot,X(\cdot,t,\cdot))$ belongs to $L^1([t-S,t+S]\times\Omega)$ and  $w^k(\cdot,X^k(\cdot,t,\cdot))\to w(\cdot,X(\cdot,t,\cdot))$ in $L^1([t-S,t+S]\times\Omega)$ as $k\to\infty$.  

2.  Let $s,t\in\R$ and $w\in L^1(K)$ be arbitrary. 
Let $\{w^k\}_{k\in\N}\subset  C^1(K)$ be such that $w^k\to w$  in $L^1(K)$ as  $k\to\infty$.    
Then, the function $w(X(s,t,\cdot))$ belongs to $L^1(\Omega)$ and  $w^k(X^k(s,t,\cdot))\to w(X(s,t,\cdot))$ in $L^1(\Omega)$ as $k\to\infty$.  
\end{Lemma}
\begin{proof}
1.  Due to  the $L^1$-convergence of $\{w^k\}$, there exists a subsequence $\{w^{a_k}\}_{k\in\N}\subset \{w^k\}$ that converges to $w$ pointwise on $([t-S,t+S]\times K)\setminus E$ with a  null set $E\subset [t-S,t+S]\times K$.   
Due to Lemma \ref{FL-image2}, there exists a null set $N\subset[t-S,t+S]\times K$ such that $E\cup N$ is Borel measurable and the inverse image $\tilde{E}:= \{ (s,x)\in\R\times\Omega\,|\,  G(s,t,x)\in E\cup N\}\subset [t-S,t+S]\times\Omega$ is a null set.  
Hence, $w^{a_k}(s,X(s,t,x))\to w(s,X(s,t,x)) $ for all $(s,x)\in ([t-S,t+S]\times\Omega)\setminus \tilde{E}$ as $k\to\infty$. Since  $w^{a_k}(\cdot,X(\cdot,t,\cdot))$ is measurable, so is $w(\cdot,X(\cdot,t,\cdot))$. 

Since $X^l(\cdot,t,\cdot)\to X(\cdot,t,\cdot)$ in $L^2([t-S,t+S]\times\Omega)^3$ as $l\to\infty$, we find a subsequence $\{X^{b_l}(\cdot,t,\cdot)\}_{l\in\N}\subset\{X^l(\cdot,t,\cdot)\}$ that converges to $X(\cdot,t,\cdot)$ a.e. pointwise on $[t-S,t+S]\times\Omega$. 
Fix any $k\gg1$ and set ${c:=e^{\norm \nabla\cdot v\norm_{L^1([t-S-\ep_1,t+S+\ep_1];L^\infty(\Omega))}}<\infty}$. Due to Proposition \ref{classical2}$|_{X=X^{b_l},w=v^{b_l},f=w^{a_k}}$ and Gronwall's inequality, we have for all $l\in\N$ and $s\in[t-S,t+S]$,
\begin{align*}
\int_\Omega |w^{a_k}(s,X^{b_l}(s,t,x))|dx\le \int_K |w^{a_k}(s,X^{b_l}(s,t,x))|dx
\le c\norm w^{a_k}(s,\cdot)\norm_{L^1(K)},
\end{align*}
which leads to 
\begin{align*}
&\int_{t-S}^{t+S}\int_\Omega |w^{a_k}(s,X^{b_l}(s,t,x))|dxds\le c\norm w^{a_k}\norm_{L^1([t-S,t+S]\times K)}\le c(1+\norm w\norm_{L^1([t-S,t+S]\times K)}).
\end{align*}
{\emergencystretch=2em
Fatou's lemma with respect to $l$ implies that $\liminf_{l\to\infty} |w^{a_k}(\cdot,X^{b_l}(\cdot,t,\cdot))|$, which is equal to $|w^{a_k}(\cdot,X(\cdot,t,\cdot))|$ a.e. on $[t-S,t+S]\times \Omega$, is integrable and\par} 
\begin{align*}
&\int_{t-S}^{t+S}\int_\Omega |w^{a_k}(s,X(s,t,x))|dxds\le\liminf_{l\to\infty}\int_{t-S}^{t+S}\int_\Omega |w^{a_k}(s,X^{b_l}(s,t,x))|{dxds}\\
&\qquad\le c(1+\norm w\norm_{L^1([t-S,t+S]\times K)}).
\end{align*}
Then, Fatou's lemma with respect to $k$ implies that $\liminf_{k\to\infty} |w^{a_k}(\cdot,X(\cdot,t,\cdot))|$, which is equal to $|w(\cdot,X(\cdot,t,\cdot))|$ on $([t-S,t+S]\times \Omega)\setminus\tilde{E}$,  is integrable. Hence,  $w(\cdot,X(\cdot,t,\cdot))$ belongs to $L^1([t-S,t+S]\times\Omega)$. 

Observe that 
\begin{align*}
&\norm w^k(\cdot,X^k(\cdot,t,\cdot)) -w(\cdot,X(\cdot,t,\cdot))\norm_{L^1([t-S,t+S]\times\Omega)}\\
&\quad \le \uwave{\norm w^k(\cdot,X^k(\cdot,t,\cdot)) -w^{a_l}(\cdot,X^k(\cdot,t,\cdot))\norm_{L^1([t-S,t+S]\times\Omega)}}_{\rm (i)} \\
&\qquad +\uwave{\norm w^{a_l}(\cdot,X^k(\cdot,t,\cdot)) -w^{a_l}(\cdot,X(\cdot,t,\cdot))\norm_{L^1([t-S,t+S]\times\Omega)}}_{\rm (ii)}\\
&\qquad +\uwave{\norm w^{a_l}(\cdot,X(\cdot,t,\cdot)) -w(\cdot,X(\cdot,t,\cdot))\norm_{L^1([t-S,t+S]\times\Omega)},}_{\rm (iii)}
\end{align*}
where $\{w^{a_l}\}_{l\in\N}$ is the a.e. pointwise convergent subsequence mentioned at the beginning of the proof.  
Let $\nu>0$ be arbitrary. 
By Proposition \ref{classical2}, we have 
\begin{align*}
{\rm (i)}\le c\norm w^k-w^{a_l}\norm_{L^1([t-S,t+S]\times K)}\le c\norm w^k-w\norm_{L^1([t-S,t+S]\times K)}+c\norm w^{a_l}-w\norm_{L^1([t-S,t+S]\times K)}.   
\end{align*}
Hence, there exists $\theta_\nu\in \N$ such that (i)$<\nu$ for all $k,l\ge \theta_\nu$.
Next we take care of (iii).  
Since $w^{a_l}(s,X(s,t,x))\to w(s,X(s,t,x)) $ for all $(s,x)\in ([t-S,t+S]\times\Omega)\setminus \tilde{E}$ as $l\to\infty$, Egorov's theorem implies that for each $\delta>0$ there exists a closed set $\Gamma_\delta\subset [t-S,t+S]\times\Omega$ such that $\meas(\Gamma_\delta)>\meas([t-S,t+S]\times\Omega)-\delta$ and  $w^{a_l}(s,X(s,t,x))\to w(s,X(s,t,x)) $ uniformly on $\Gamma_\delta$ as $l\to\infty$. 
On the other hand, setting $\Gamma_\delta^\circ:=[t-S,t+S]\times\Omega\setminus \Gamma_\delta$, we see that for any $m\in\N$,
\begin{align*}
\iint_{\Gamma_\delta^\circ} |w^{a_l}(s,X^m(s,t,x))|dxds
&=\iint_{\{G^m(s,t,x)\,|\,(s,x)\in \Gamma_\delta^\circ\}} |w^{a_l}(s,y)| |\det \p_yX^m(t,s,y) |  dyds\\
&\le c\iint_{\{G^m(s,t,x)\,|\,(s,x)\in \Gamma_\delta^\circ\}} |w^{a_l}(s,y)|dyds\\
&\le c\iint_{\{G^m(s,t,x)\,|\,(s,x)\in \Gamma_\delta^\circ\}} |w(s,y)|dyds+c\norm w^{a_l}-w\norm_{L^1([t-S,t+S]\times{K})},
\end{align*}
where Jacobi's formula $\frac{d}{dt}(\det \p_yX^m(t,s,y))=(\det \p_yX^m(t,s,y))(\nabla\cdot v^m(t,{X^m(t,s,y)}))$ is used. 
Since 
$$\meas(\{G^m(s,t,x)\,|\,(s,x)\in \Gamma_\delta^\circ\})\le c\meas ( \Gamma_\delta^\circ) <c\delta$$
 due to Proposition \ref{classical3},   absolute continuity of the integral of $|w|$ implies that there exists $0<\delta\ll1$  such that 
 \begin{align*}
\iint_{\Gamma_\delta^\circ} |w^{a_l}(s,X^m(s,t,x))|dxds <2\nu \mbox{\quad for all $l\ge \theta_\nu$ and $m\in \N$}. 
\end{align*}
Then, Fatou's lemma applied to $\iint_{\Gamma_\delta^\circ} |w^{a_l}(s,X^{b_m}(s,t,x))|dxds$ with respect to $m$ yields 
$$\iint_{\Gamma_\delta^\circ} |w^{a_l}(s,X(s,t,x))|dxds\le2 \nu  \mbox{\quad for all $l\ge \theta_\nu$}.$$
Absolute continuity of the integral of $w(\cdot,X(\cdot,t,\cdot))$ implies that there exists  $0<\delta\ll1$ such that 
\begin{align*}
\iint_{\Gamma_\delta^\circ} |w(s,X(s,t,x))|dxds <\nu. 
\end{align*}
Fix such  $\delta>0$. 
The uniform convergence of  $w^{a_l}(s,X(s,t,x))$ to $w(s,X(s,t,x))$ on $\Gamma_\delta$ as $l\to\infty$ implies that there exists $l\ge \theta_\nu$ such that  
$$\iint_{\Gamma_\delta} |w^{a_l}(s,X(s,t,x))- w(s,X(s,t,x))|dxds<\nu,$$
which leads to (iii)$<4\nu$.  
With this  $l$ fixed, we see that (ii)$\to 0$ as $k\to\infty$, since \(w^{a_l}\) is Lipschitz continuous with respect to the space variable{\footnote{Note also that \(K\) is convex. The line segment joining \(X^k(r,t,x)\) and \(X(r,t,x)\) lies in a fixed compact ball contained in \(K\), where the spatial derivatives of \(w^{a_l}\) are bounded.}} and $\norm X^k(\cdot,t,\cdot)- X(\cdot,t,\cdot)\norm_{L^2([t-S,t+S]\times\Omega)}\to 0$ as $k\to\infty$.  
Thus, we see that there exists $k_\nu>0$ such that for all $k\ge k_\nu$ we have 
$$\norm w^k(\cdot,X^k(\cdot,t,\cdot)) -w(\cdot,X(\cdot,t,\cdot))\norm_{L^1([t-S,t+S]\times\Omega)}< 6\nu.$$
 This concludes the assertion 1. 

2. We may follow the same reasoning. 
\end{proof}
\begin{Prop}\label{FL-preservative}
For each $t\in\R$, the map $X(\cdot,t,\cdot)$ satisfies   
\begin{align}\label{w-ODE}
\int_\R\int_\Omega \Big\{X(s,t,x)-x-\int_t^s v(r,X(r,t,x))dr\Big\}\varphi(s,x)dsdx=0\quad\mbox{for all $\varphi\in C^\infty_0(\R\times \Omega)$}.
\end{align}
Furthermore, it holds that  for each  $s,t,\tau\in\R$,
\begin{align*}
 &X(s,t, X(t,s,x))=x,\quad  X(s,t,x)=X(s,\tau, X(\tau,t,x)) \quad \mbox{for  a.e. $x\in\Omega$}.
\end{align*}
\end{Prop}
\begin{proof}
Take any  $\varphi\in C^\infty_0(\R\times \Omega)$. With $S>0$ such that supp$(\varphi)\subset[t-S,t+S]\times\Omega$, it holds that  $\norm X^k(\cdot,t,\cdot)- X(\cdot,t,\cdot)\norm_{L^2([t-S,t+S]\times\Omega)}\to 0$ as $k\to\infty$. 
Recall that for each $(s,x)\in [t-S,t+S]\times(K\setminus\overline{\Omega})$ we  have $X^k(s,t,x)\to x$ as $k\to\infty$. 
Hence, with the extension $X(s,t,x):=x$ outside  ${\Omega}$, we see that  $\norm X^k(\cdot,t,\cdot)- X(\cdot,t,\cdot)\norm_{L^1([t-S,t+S]\times K)}\to 0$ as $k\to\infty$ and we may apply Lemma \ref{FL-product} with $w^k=v^k_i$ ($i=1,2,3$).  
It follows from  \eqref{flow^k} that 
\begin{align*}
\int_\R\int_\Omega \Big\{X^k(s,t,x)-x-\int_t^s v^k(r,X^k(r,t,x))dr\Big\}\varphi(s,x)dsdx=0.
\end{align*}
We send  $k\to\infty$ to obtain \eqref{w-ODE}. 

It follows from 2.\,of Lemma \ref{FL-product} with $w^k=X^k(s,t,\cdot)$ (we need extension outside ${\Omega}$ as above) that  $X(s,t,X(t,s,\cdot))$ belongs to $L^2(\Omega)$, while the equality $X^k(s,t,X^k(t,s,x))=x$ for all $k$ yields  $X(s,t,X(t,s,x))=x$. 
 Similarly, $X^k(s,t,x)=X^k(s,\tau,X^k(\tau,t,x))$ provides the last assertion.
\end{proof}
\begin{Prop}\label{FL2}
Let $t\in\R$ be arbitrary. 
There exists a null set $E_\Omega\subset\Omega$ such that for each fixed $x\in\Omega\setminus E_\Omega$ the function $v(\cdot,X(\cdot,t,x))$ is locally integrable  and there exists a null set $N(x)\subset\R$ for which  
$$\mbox{$\dis X(s,t,x)=x+\int_t^s v(r,X(r,t,x))dr$ for all $s\in \R\setminus N(x)$},$$ where  there exists an absolutely continuous curve $\gamma(\cdot)$ that is equal to $X(\cdot,t,x)$  on $ \R\setminus N(x)$ and satisfies  $\gamma(s)=x+\int_t^s v(r,\gamma(r))dr$ for all $s\in\R$. 
\end{Prop} 
\begin{proof}
Since $v(\cdot,X(\cdot,t,\cdot))$ is $s$-locally integrable on $\R\times\Omega$, Fubini's theorem implies that there exists a null set $E^0_\Omega\subset\Omega$ such that $v(\cdot,X(\cdot,t,x))$ is $s$-locally integrable for each fixed $x\in \Omega\setminus E^0_\Omega$.  
Due to Proposition \ref{FL-preservative}, we have  
$$\Big|X(s,t,x)-x-\int_t^s v(r,X(r,t,x))dr\Big|=0\quad\mbox{for a.e. $(s,x)\in \R\times(\Omega\setminus E^0_\Omega)$}.$$
 By Fubini's theorem, there exists a null set $E^1_\Omega\subset(\Omega\setminus E^0_\Omega)$ such that  for each fixed $x\in\Omega\setminus (E^1_\Omega\cup E^0_\Omega)$ we have $|X(s,t,x)-x-\int_t^s v(r,X(r,t,x))dr|=0$ for a.e. $s\in\R$, i.e.,  there exists a null set $N(x)\subset\R$ such that  
 $$X(s,t,x)=x+\int_t^s v(r,X(r,t,x))dr\quad\mbox{for all $s\in \R\setminus N(x) $.}$$
 Since $v(\cdot,X(\cdot,t,x))$ is locally integrable, $\gamma: {\R}\to\R^3$, $\gamma(s):=x+\int_t^s v(r,X(r,t,x))dr$ is  absolutely continuous and  $\gamma(s)=X({s},t,x)$ for all $s\in \R\setminus N(x)$.  
 Hence, we have $v(\cdot,X(\cdot,t,x))=v(\cdot,\gamma(\cdot))$ for all $s\in \R\setminus N(x)$ and thus  $\gamma(s)=x+\int_t^s v(r,\gamma(r))dr$ for all $s\in\R$, or $\gamma'(s)=v(s,\gamma(s))$ a.e. $s\in\R$ with $\gamma(t)=x$.  
\end{proof}
\noindent Before proving the next proposition, we make explicit a consequence of \eqref{pre1-2}. For each fixed $s,t\in\R$, we have
\begin{align*}
&X(s,t,\Omega\setminus E^{s,t})=\Omega\setminus E^{t,s},\qquad
\bigl(X(s,t,\cdot)|_{\Omega\setminus E^{s,t}}\bigr)^{-1}
=X(t,s,\cdot)|_{\Omega\setminus E^{t,s}}.
\end{align*}
Indeed, let $x\in\Omega\setminus E^{s,t}$ and set $y=X(s,t,x)$. Then $X(t,s,y)=x$. Since $X$ is extended by the identity outside $\Omega$, we have $y\in\Omega$. Moreover,
$$X(s,t,X(t,s,y))=X(s,t,x)=y,$$
so that $y\in\Omega\setminus E^{t,s}$. Interchanging $s$ and $t$ gives the reverse inclusion. Thus, the restrictions of $X(s,t,\cdot)$ and $X(t,s,\cdot)$ to these two sets are inverse bijections. The exceptional sets are null by Proposition \ref{FL-preservative}.

\medskip
\begin{Prop}[Common trimming]\label{FL-common}
For each $s,t\in\R$, there exists a null set $\dot{N}^{s,t}\subset\Omega$ containing $E^{s,t}$ such that $X(s,t,A\setminus\dot{N}^{s,t})$ is Borel measurable for every Borel measurable set $A\subset\Omega$.
\end{Prop}
\begin{proof}
Fix $s,t\in\R$. We repeat the proof of assertion 4 of Lemma \ref{FL-image}, with $\Omega\setminus E^{s,t}$ in place of $B$. Thus, there exist a subsequence $\{X^{c_k}\}\subset\{X^k\}$ and closed sets $J_l\subset\Omega\setminus E^{s,t}$, $l\in\N$, such that
$$\meas(\Omega\setminus J_l)<\frac1l,\qquad
X^{c_k}(s,t,\cdot)\to X(s,t,\cdot)\quad\mbox{uniformly on $J_l$ as $k\to\infty$}.$$
Set $J^\ast:=\bigcup_{l\in\N}J_l$ and $\dot{N}^{s,t}:=\Omega\setminus J^\ast$. Then $\dot{N}^{s,t}$ is a null set containing $E^{s,t}$, and its choice does not involve $A$.
Since \(J_l\) is compact and \(X(s,t,\cdot)\) is continuous and injective on \(J_l\), its restriction \(F_l:=X(s,t,\cdot)|_{J_l}\) is a homeomorphism onto the compact set \(X(s,t,J_l)\). Hence, for every Borel measurable set \(A\subset\Omega\),
\[
X(s,t,A\cap J_l)=(F_l^{-1})^{-1}(A\cap J_l)
\]
is Borel measurable. 
Consequently,
$$X(s,t,A\setminus\dot{N}^{s,t})=\bigcup_{l\in\N}X(s,t,A\cap J_l)$$
is Borel measurable. This completes the proof.
\end{proof}
\medskip

\begin{proof}[Proof of Theorem \ref{Thm-DL}.] The assertion 1 is confirmed by Proposition \ref{LT-kasoku-2} in the notation $X^k,X$ and Lemma \ref{FL-product}. 
The assertion 2 is confirmed by Proposition  \ref{FL-preservative} and \eqref{RTT1}\footnote{Although \eqref{RTT1} is placed in Theorem  \ref{FL-preservative2},  \eqref{RTT1}  is proved only by Lemma  \ref{FL-product} as seen below, and there is no circularity.} with the upper and lower estimates of $\meas(X^k(t,s,A))$ based on \eqref{usi2}.  
 The assertion 3 is given in Proposition \ref{FL2}.     \end{proof}
\begin{proof}[Proof of Theorem \ref{FL-preservative2}]
 For the indicator function $\chi_A(\cdot)$ of $A$, we have 
 \begin{align*}
&\int_K\chi_A(X^k(s,t,x))dx =\int_{X^k(s,t,\cdot)^{-1}( A)}dx
 =\meas(X^k(s,t,\cdot)^{-1}(A))=\meas(X^k(t,s,A)),\\
 &\int_{K\setminus\Omega}\chi_A(X^k(s,t,x))dx
 =  \int_{\{ x\in K \setminus \Omega\, |\, X^k(s,t,x)\in A \}}dx
 =\meas(X^k(t,s,A)\setminus\Omega),\\
 & \int_\Omega\chi_A(X^k(s,t,x))dx=\meas(X^k(t,s,A)) - \meas(X^k(t,s,A)\setminus\Omega).
\end{align*}
For each $\ep>0$, there exists an open set $\Omega^\ep\supset\overline{\Omega}$ such that $\meas(\Omega^\ep)<\meas(\overline{\Omega})+\ep=\meas(\Omega)+\ep$ ($\Omega$ is assumed to satisfy $\meas(\overline{\Omega})=\meas(\Omega)$). Since supp$(v^k)\subset \R\times\Omega^\ep$ for all sufficiently large $k$,  we have $X^k(t,s,A)\setminus\Omega\subset \Omega^\ep\setminus\Omega$ and $\meas(X^k(t,s,A)\setminus\Omega)<\ep$, i.e., $\meas(X^k(t,s,A)\setminus\Omega)\to0$ as $k\to\infty$.   
 Due to Lemma \ref{FL-product} with $w=\chi_A$, the integral $ \int_\Omega\chi_A(X(s,t,x))dx=\meas(X(s,t,\cdot)^{-1}( A))$ makes sense, and with the mollification $\chi_A^{k}$ of $\chi_A$ by the mollifier $\eta^{k^{-1}}(\cdot)$, we have 
 \begin{align*}
 &\Big|\int_\Omega\chi_A(X^k(s,t,x))dx-\meas(X(s,t,\cdot)^{-1}( A))\Big|
 \le \norm \chi_A\circ X^k(s,t,\cdot)-\chi_A\circ X(s,t,\cdot)\norm_{L^1(\Omega)}\\
&\qquad \le  \norm \chi_A\circ X^k(s,t,\cdot)-\chi^k_A\circ X^k(s,t,\cdot)\norm_{L^1(\Omega)}
+\norm \chi_A^k\circ X^k(s,t,\cdot)-\chi_A\circ X(s,t,\cdot)\norm_{L^1(\Omega)},\\
& \norm \chi_A\circ X^k(s,t,\cdot)-\chi^k_A\circ X^k(s,t,\cdot)\norm_{L^1(\Omega)}\le c\norm \chi^k_A-\chi_A\norm_{L^1(K)}\to0\mbox{\quad as $k\to\infty$},\\ 
&\norm \chi_A^k\circ X^k(s,t,\cdot)-\chi_A\circ X(s,t,\cdot)\norm_{L^1(\Omega)} \to0\mbox{\quad as $k\to\infty$},
 \end{align*}
 where ${c:=e^{\norm \nabla\cdot v\norm_{L^1([\min(s,t)-\ep_1,\max(s,t)+\ep_1];L^\infty(\Omega))}}<\infty}$.
Therefore, we obtain  \eqref{RTT1}.
 
 By the assertion 5 of Lemma \ref{FL-image}, {with $B=A$ and $N_B^{s,t}=A\cap\dot{N}^{s,t}$,} there exist a subsequence $\{X^{a_k}\}\subset\{X^k\}$ for which we find for each $\delta>0$ and $\ep>0$ a closed set $A_\delta\subset \R^3$, open set $O_\ep\subset K$ and $k_\ep\in\N$ such that  
\begin{align}\label{eiji}
\begin{cases}
&A_\delta\subset A\setminus {\dot{N}^{s,t}}\quad \mbox{with  $\meas(A_\delta)>\meas(A)-\delta$},\\
 & \meas(O_\ep)< \meas(X(s,t,A_\delta) )+\ep,\\
& X(s,t,A_\delta) \mbox{ is a closed subset of $O_\ep$},\\
& X^{a_k}(s,t,A_\delta) \mbox{ is a closed subset of $ O_\ep$ for all $k\ge k_\ep$},\\
&\meas(X^{a_k}(s,t, A_\delta ))-\ep<  \meas( X(s,t,A\setminus {\dot{N}^{s,t}}))\quad\mbox{for all $k\ge k_\ep$}.
\end{cases}
\end{align}
Due to injectivity of $X^{a_k}$ and Proposition \ref{classical3}, we have 
\begin{align*}
&\meas(X^{a_k}(s,t, A_\delta ))=\meas(X^{a_k}(s,t, A ))-\meas(X^{a_k}(s,t, A\setminus A_\delta )),\\
&\meas(X^{a_k}(s,t, A\setminus A_\delta ))=\meas(A\setminus A_\delta)+ \int_t^s\int_{X^{a_k}(r,t,A\setminus A_\delta)} \nabla\cdot v^{a_k}(r,x)dxdr\le c\,\,\meas(A\setminus A_\delta)<c\delta.
\end{align*} 
 Since we already know from \eqref{RTT1} that $\{\meas(X^k(s,t, A))\}_{k\in\N}$ is convergent, we obtain 
 \begin{align*}
 \lim_{k\to\infty}\meas(X^{k}(s,t, A)) -c\delta-\ep&=\lim_{k\to\infty}\meas(X^{a_k}(s,t, A)) -c\delta-\ep\\
 &\le \meas( X(s,t,A\setminus {\dot{N}^{s,t}}))\quad\mbox{ for any $\delta>0,\,\,\ep>0$}.
 \end{align*}
We show the estimate from above.  
 Recall that $X(t,s,X(s,t,x))=x$ for all $x\in\Omega\setminus E^{s,t}$.  
For each $y\in  X(s,t,A\setminus {\dot{N}^{s,t}})$, there exists $x_y\in A\setminus {\dot{N}^{s,t}}\subset A\setminus E^{s,t}$ such that $ X(s,t,x_y)=y$. 
 Hence, we have $X(t,s,y)=X(t,s,X(s,t,x_y))=x_y\in A$, which means that $y\in X(t,s,\cdot)^{-1}(A)$ and $ X(s,t,A\setminus {\dot{N}^{s,t}})\subset X(t,s,\cdot)^{-1}(A)$ to obtain \eqref{RTT5}. 
Therefore, with \eqref{RTT1}, we see that  $\meas( X(s,t,A\setminus {\dot{N}^{s,t}}))\le \meas(X(t,s,\cdot)^{-1}(A))=\lim_{k\to\infty} \meas(X^k(s,t,A))$ to conclude \eqref{RTT2}.  

Our proof of \eqref{RTT3} consists of three steps. As Step 1, we prove using \eqref{eiji}, 
\begin{align}\label{kazu1}
 \alpha:=\int_{X(s,t,A\setminus {\dot{N}^{s,t}})} f(x)dx= \lim_{k\to\infty}  \int_{X^{a_k}(s,t,A)} f(x)dx.
\end{align}
With \eqref{eiji}, observe that 
\begin{align*}
&\int_{X^{a_k}(s,t,A)} f(x)dx-\int_{X(s,t,A\setminus {\dot{N}^{s,t}})} f(x)dx \\
&\quad= \uwave{\int_{X^{a_k}(s,t,A)} f(x)dx-\int_{X^{a_k}(s,t,A_\delta)} f(x)dx}_{\rm(i)}+\uwave{\int_{X^{a_k}(s,t,A_\delta)} f(x)dx-\int_{X(s,t,A\setminus {\dot{N}^{s,t}})} f(x)dx,}_{\rm (ii)} \\
&{\rm(i)}=\int_{X^{a_k}(s,t,A\setminus A_\delta)} f(x)dx,\\
&{\rm(ii)}= \Big(\int_{X^{a_k}(s,t,A_\delta)} -\int_{O_\ep} +\int_{O_\ep} -\int_{X(s,t,A_\delta)} +\int_{X(s,t,A_\delta)}- \int_{X(s,t,A\setminus {\dot{N}^{s,t}})}\Big) f(x)dx\\
&\quad = -\int_{O_\ep\setminus X^{a_k}(s,t, A_\delta)}  \!\!\!\!\!f(x)dx+\int_{O_\ep\setminus X(s,t, A_\delta)}  \!\!\!\!\!f(x)dx
-  \int_{X(s,t,A\setminus {\dot{N}^{s,t}})\setminus X(s,t,A_\delta)} \!\!\!\!\!f(x)dx\quad  \mbox{for all $k\ge k_\ep$}.
\end{align*} 
Set $J_\theta(f):=\sup \{\int_B|f|dx \,:\,B\subset K,\,\,\,\meas(B)\le \theta\}$ with $\theta\ge0$, where $ J_\theta(f)\to0$ as $\theta\to0$ due to absolute continuity of the integral of $|f|$.  
Since $\meas(X^{a_k}(s,t,A\setminus A_\delta))\le c \,\,\meas(A\setminus A_\delta)<c\delta$ for all $k\in\N$, we have $|\mbox{(i)}|<J_{c\delta}(f)$ for all $k\in\N$. 
{Since $A_\delta\subset A\setminus\dot{N}^{s,t}$, we have $A_\delta\setminus\dot{N}^{s,t}=A_\delta$. Thus, the already proved \eqref{RTT2}, applied to $A_\delta$ with the same common trimming, yields $\meas(X(s,t,A_\delta))=\lim_{k\to\infty}\meas(X^{a_k}(s,t,A_\delta))$.}
Hence, we see that $\meas(O_\ep\setminus X^{a_k}(s,t,A_\delta))\to \meas(O_\ep)-\meas(X(s,t,A_\delta))<\ep$, which implies that there exists $\tilde{k}_\ep\ge k_\ep$ such that if $k\ge \tilde{k}_\ep$ we have   $\meas(O_\ep\setminus X^{a_k}(s,t,A_\delta))<2\ep$ and $|\int_{O_\ep\setminus X^{a_k}(s,t, A_\delta)} f(x)dx|\le J_{2\ep}(f)$.  
It is clear that $|\int_{O_\ep\setminus X(s,t, A_\delta)} f(x)dx|\le J_\ep(f)$. 
As we already confirmed, we have
\begin{align*}
&\meas(X(s,t,A\setminus {\dot{N}^{s,t}})\setminus X(s,t,A_\delta))=\meas(X(s,t,A\setminus {\dot{N}^{s,t}}))-\meas( X(s,t,A_\delta))\\
 &=\lim_{k\to\infty} \meas(X^{a_k}(s,t,A)) -\lim_{k\to\infty} \meas(X^{a_k}(s,t,A_\delta)) =\lim_{k\to\infty} \meas(X^{a_k}(s,t,A\setminus A_\delta))\le c\delta. 
 \end{align*}
 Therefore, we obtain $|\int_{X(s,t,A\setminus {\dot{N}^{s,t}})\setminus X(s,t,A_\delta)} f(x)dx|\le J_{c\delta}(f)$. 
To sum up, for any $\ep>0,\,\,\delta>0$, it holds that $|\mbox{(i)}|+|\mbox{(ii)}|<2J_{c\delta}(f)+J_{2\ep}(f)+J_{\ep}(f)$ for all $k\ge\tilde{k}_\ep$.  
Thus, we conclude \eqref{kazu1}.

As Step 2, we prove
\begin{align}\label{kazu2}
 \int_{X(t,s,\cdot)^{-1}(A)} f(x)dx= \lim_{k\to\infty}  \int_{X^{a_k}(t,s,\cdot)^{-1}(A)} f(x)dx= \lim_{k\to\infty}  \int_{X^{a_k}(s,t,A)} f(x)dx=\alpha.
\end{align}
Due to \eqref{RTT2}, we have $\meas(X(s,t,A\setminus {\dot{N}^{s,t}}))=\meas(X(t,s,\cdot)^{-1}(A))$. 
Furthermore,   by \eqref{RTT5}, we have $ X(s,t,A\setminus {\dot{N}^{s,t}})\subset X(t,s,\cdot)^{-1}(A)$ and $\meas(X(t,s,\cdot)^{-1}(A) \setminus X(s,t,A\setminus {\dot{N}^{s,t}}))=0$. 
Hence,  by \eqref{kazu1}, we obtain   
\begin{align*}
\int_{X^{a_k}(t,s,\cdot)^{-1}(A)} f(x)dx&=\int_{X(s,t,A\setminus {\dot{N}^{s,t}})}f(x)dx+  \int_{X^{a_k}(s,t,A)} f(x)dx-\int_{X(s,t,A\setminus {\dot{N}^{s,t}})}f(x)dx,\\
\int_{X^{a_k}(s,t,A)} f(x)dx&-\int_{X(s,t,A\setminus {\dot{N}^{s,t}})}f(x)dx\to0\quad \mbox{as $k\to\infty$},\\
\int_{X(s,t,A\setminus {\dot{N}^{s,t}})}f(x)dx&
= \int_{X(t,s,\cdot)^{-1}(A)}f(x)dx-  \int_{X(t,s,\cdot)^{-1}(A) \setminus X(s,t,A\setminus {\dot{N}^{s,t}})  }f(x)dx\\
&= \int_{X(t,s,\cdot)^{-1}(A)}f(x)dx,
\end{align*}
which confirms \eqref{kazu2}. 
The key to the final step is that $\alpha$ turns out to be independent of  the choice of subsequence $\{X^{a_k}\}$. 

 As Step 3, suppose that  \eqref{RTT3} does not hold. 
Then, there exists $\nu>0$ such that for each $k\in\N$ we find $b_k\ge k$ for which $|\int_{X^{b_k}(s,t,A)}f(x)dx-\alpha|\ge\nu$ for all $k\in\N$.  
Since $X^{b_k}(s,t,\cdot)\to X(s,t,\cdot)$ in $L^2(\Omega)^3$ as $k\to\infty$, we may use the assertion 5 of Lemma \ref{FL-image} with $\{X^{b_k}\}$ in place of $\{X^k\}$ and follow the reasoning for \eqref{kazu1} and \eqref{kazu2}, to confirm that  there exists a subsequence $\{X^{c_k}\}\subset\{X^{b_k}\}$ such that 
\begin{align*}
\alpha= \int_{X(t,s,\cdot)^{-1}(A)} f(x)dx= \lim_{k\to\infty}  \int_{X^{c_k}(t,s,\cdot)^{-1}(A)} f(x)dx= \lim_{k\to\infty}  \int_{X^{c_k}(s,t,A)} f(x)dx.
\end{align*}
Since  $|\int_{X^{c_k}(s,t,A)}f(x)dx-\alpha|\ge\nu>0$ for all $k\in\N$, we reach a contradiction and complete the proof of  \eqref{RTT3}.

Finally, we prove \eqref{RTT4}. Due to \eqref{RTT2}, we have  $\meas(X(t,s,A\setminus {\dot{N}^{t,s}}))=\meas(X(s,t,\cdot)^{-1}(A))$. 
Furthermore,   by \eqref{RTT5}, we have $ X(t,s,A\setminus {\dot{N}^{t,s}})\subset X(s,t,\cdot)^{-1}(A)$ and $\meas(X(s,t,\cdot)^{-1}(A) \setminus X(t,s,A\setminus {\dot{N}^{t,s}}))=0$. 
Hence,  by \eqref{RTT3}, we obtain   
\begin{align*}
\int_{X^k(s,t,\cdot)^{-1}(A)} f(x)dx&=\int_{X(t,s,A\setminus {\dot{N}^{t,s}})}f(x)dx+  \int_{X^k(t,s,A)} f(x)dx-\int_{X(t,s,A\setminus {\dot{N}^{t,s}})}f(x)dx,\\
\int_{X^k(t,s,A)} f(x)dx&-\int_{X(t,s,A\setminus {\dot{N}^{t,s}})}f(x)dx\to0\quad \mbox{as $k\to\infty$},\\
\int_{X(t,s,A\setminus {\dot{N}^{t,s}})}f(x)dx&
= \int_{X(s,t,\cdot)^{-1}(A)}f(x)dx-  \int_{X(s,t,\cdot)^{-1}(A) \setminus X(t,s,A\setminus {\dot{N}^{t,s}})  }f(x)dx\\
&= \int_{X(s,t,\cdot)^{-1}(A)}f(x)dx,
\end{align*}
which confirms  \eqref{RTT4}.
\end{proof}



\medskip
\noindent {\bf Remark.} {\it By essential invertibility \eqref{pre1-2} and Lemma \ref{FL-image2}, which shows that inverse images of null sets are null, we can also obtain the Lebesgue (not necessarily Borel) measurability of trimmed forward images using the backward flow. Indeed, let $N\subset\Omega$ be any null set containing $E^{s,t}$. For every Borel measurable set $A\subset\Omega$, the inverse identities give
\begin{align*}
&X(s,t,A\setminus N)
=(\Omega\setminus E^{t,s})\cap X(t,s,\cdot)^{-1}(A\setminus N),\qquad 
X(s,t,A\setminus N)\subset X(t,s,\cdot)^{-1}(A),\\
&X(t,s,\cdot)^{-1}(A)\setminus X(s,t,A\setminus N)
\subset E^{t,s}\cup X(t,s,\cdot)^{-1}(N).
\end{align*}
The right hand side of the last inclusion is a null set by Proposition \ref{FL-preservative} and Lemma \ref{FL-image2}. Thus, $X(s,t,A\setminus N)$ is Lebesgue measurable and agrees with $X(t,s,\cdot)^{-1}(A)$ up to a null set.

This is consistent with the essential invertibility of regular Lagrangian flows discussed in Ambrosio-Crippa \cite{AC}; see also Zhang \cite[Lemma 3.4]{Zhang} for a criterion for essential invertibility of measurable maps. In contrast, Proposition \ref{FL-common} establishes Borel measurability for the common trimming directly by Egorov\'s theorem.}\medskip

\begin{proof}[Proof of Theorem \ref{main}.] 
Due to  classical Reynolds transport theorem, it holds that
\begin{align}\label{rei}
\int_{X^k(s,t,A)}g(s,x)dx=&\int_{A}g(t,x)dx+\int_t^s\int_{X^k(r,t,A)}\left\{\frac{\partial g}{\partial{r}}(r,x)+\nabla\cdot\Big(g(r,x)v^k(r,x)\Big)\right\}dxdr\\\nonumber
=&\int_{A}g(t,x)dx+\int_t^s\int_{X^k(r,t,A)}\left\{\frac{\partial g}{\partial{r}}(r,x)+\nabla\cdot\Big(g(r,x)v(r,x)\Big)\right\}dxdr\\\nonumber
&+\int_t^s\int_{X^k(r,t,A)}\left\{\nabla\cdot\Big(g(r,x)v^k(r,x)\Big)-\nabla\cdot\Big(g(r,x)v(r,x)\Big)\right\}dxdr,
\end{align}
where we note that the graph of $X^k(r,t,A)$ with respect to $r$ is measurable\footnote{The continuous maps $F,\tilde{F}:\R^4\to\R^4$, $F(r,x):=(r,X^k(r,t,x))$, $\tilde{F}(r,x):=(r,X^k(t,r,x))$ are bijective and $\tilde{F}(F(I\times A))=I\times A$ for all intervals $I\subset\R$. Hence, $F(I\times A)=\tilde{F}^{-1}(I\times A)$ is Borel measurable.} in $\R^4$; Fubini's theorem implies that $r\mapsto \int_{X^k(r,t,A)}\{\frac{\partial g}{\partial{r}}(r,x)+\nabla\cdot(g(r,x)v(r,x))\}dx$ is measurable; the last term on the right hand side vanishes as $k\to\infty$.  
There exists a null set $N\subset I^{s,t}$ ($=[t,s]$ or $[s,t]$) for which $v(r,\cdot)\in L^1(\Omega)^3, \nabla\cdot v(r,\cdot)\in L^1(\Omega)$,  and  for all $k\in\N$ and each $r\in I^{s,t}\setminus N$,  
\begin{align*}
\Big|\int_{X^k(r,t,A)}\left\{\frac{\partial g}{\partial{r}}(r,x)+\nabla\cdot\Big(g(r,x)v(r,x)\Big)\right\}dx\Big|\le \Big|\!\Big| \frac{\partial g}{\partial{r}}(r,\cdot)+\nabla\cdot\Big(g(r,\cdot)v(r,\cdot)\Big)\Big|\!\Big|_{L^1(K)},
\end{align*}
where the {right} hand side is integrable with respect to $r$; due to  \eqref{RTT3},  we have for each $r\in I^{s,t}\setminus N$,
$$\lim_{k\to\infty} \int_{X^k(r,t,A)}\left\{\frac{\partial g}{\partial{r}}(r,x)+\nabla\cdot\Big(g(r,x)v(r,x)\Big)\right\}dx= \int_{X(r,t,A\setminus {\dot{N}^{r,t}})}\left\{\frac{\partial g}{\partial{r}}(r,x)+\nabla\cdot\Big(g(r,x)v(r,x)\Big)\right\}dx,$$
from which we confirm that  the map 
$r\mapsto \int_{X(r,t,A\setminus {\dot{N}^{r,t}})}\left\{\frac{\partial g}{\partial{r}}(r,x)+\nabla\cdot(g(r,x)v(r,x))\right\}dx$ 
 is measurable. {This argument does not require a measurable choice of the sets $\dot{N}^{r,t}$ with respect to $r$; the integrals are measurable as limits of the corresponding integrals for $X^k$.}
Sending $k\to\infty$ in \eqref{rei} with Lebesgue's dominated convergence  theorem, we obtain  \eqref{trans2}. 
 \eqref{trans3} follows from the same reasoning with  \eqref{RTT4}. 
 \eqref{trans1}, \eqref{trans0} are obtained by choosing $g\equiv1$.  
\end{proof}


\noindent{\bf Acknowledgements.}
The author is grateful to Norihisa Ikoma (Keio University, Japan) for helpful  discussion on the proof of the assertion 1 of Lemma \ref{FL-image} and to 
 Dieter Bothe (Technische Universit\"at Darmstadt, Germany)   for sharing the work \cite{BK} with valuable discussions.  

\noindent{\bf Funding declaration} The author is supported by JSPS Grants-in-Aid for Scientific Research (C) \#22K03391.

\noindent{\bf Data availability.}
Data sharing not applicable to this article as no datasets were generated or analyzed during the current study.

\noindent{\bf Conflicts of interest statement.} The author states that there is no conflict of interest.
 
\end{document}